\documentclass[a4paper,reqno]{amsart}
\usepackage{latexsym,amsmath,amssymb,amsfonts}
\usepackage{a4wide}

\catcode`@=11 
\def\blfootnote{\xdef\@thefnmark{}\@footnotetext}
\catcode`@=12 

\theoremstyle{plain}
\begingroup
\newtheorem{theorem}{Theorem}[section]
\newtheorem{lemma}[theorem]{Lemma}
\newtheorem{proposition}[theorem]{Proposition}
\newtheorem{corollary}[theorem]{Corollary}
\endgroup

\theoremstyle{definition}
\begingroup
\newtheorem{definition}[theorem]{Definition}
\newtheorem{remark}[theorem]{Remark}

\endgroup

\theoremstyle{remark}

\mathchardef\emptyset="001F

\numberwithin{equation}{section}

\newcommand{\dist}{{\rm dist\,}}
\newcommand{\e}{\varepsilon}
\newcommand{\Om}{\Omega}

\newcommand{\G}{\Gamma}

\newcommand{\R}{{\mathbb R}}
\newcommand{\N}{\mathbb N}

\newcommand{\wto}{\rightharpoonup}
\renewcommand{\div}{{\rm div}}
\newcommand{\LL}{{\mathcal L}}
\newcommand{\A}{{\mathcal A}}
\newcommand{\Ar}{{\mathcal A}_{reg}}
\newcommand{\setmeno}{\!\setminus\!}
\newcommand{\hn}{{\mathcal H}^{N-1}}
\newcommand{\vf}{v_\varphi}

\newcommand{\vP}{v_{\vartheta,\Phi}}
\newcommand{\nablat}{\nabla_\Gamma}
\newcommand{\Deltat}{\Delta_\Gamma}
\newcommand{\divt}{\div_\Gamma}
\newcommand{\Gam}{{\Gamma\cap U}}
\newcommand{\tu}{T}
\newcommand{\supp}{{\rm supp\,}}
\newcommand{\B}{{\mathbf B}}

\newcommand{\capset}{\, {\subset{\hskip-7pt\cdot}\ }}

\title[A second order minimality condition for the Mumford-Shah functional]
{A second order minimality condition for the Mumford-Shah functional}
\author{F.\ Cagnetti}
\author{M.G.\ Mora}
\author{M.\ Morini}
\address[F.~Cagnetti, M.G.~Mora, and M.~Morini]{SISSA, Via Beirut 2, 
34014 Trieste, Italy}
\email[Filippo Cagnetti]{cagnetti@sissa.it}
\email[Maria Giovanna Mora]{mora@sissa.it}
\email[Massimiliano Morini]{morini@sissa.it}

\begin{document}

\begin{abstract}
A new necessary minimality condition for the Mumford-Shah functional is derived by means
of second order variations. It is expressed in terms of a sign condition for a nonlocal quadratic
form on $H^1_0(\Gamma)$, $\Gamma$ being a submanifold of the regular part of the discontinuity
set of the critical point. Two equivalent formulations are provided: one in terms of the first eigenvalue 
of a suitable compact operator, the other involving a sort of nonlocal capacity of $\Gamma$.
A sufficient condition for minimality is also deduced. Finally, an explicit example is discussed, where
a complete characterization of the domains where the second variation is nonnegative can be given. 
\end{abstract}
\maketitle

{\small

\bigskip

}
\bigskip
\bigskip

\begin{section}{Introduction}

The subject of this paper is the derivation and the analysis of a new
minimality condition for the Mumford-Shah functional, obtained by means of second order variations.

The homogeneous Mumford-Shah functional on a Lipschitz domain $\Om$ in $\R^N$, $N\geq2$, is defined~as
\begin{equation}\label{MS}
F(u,K)=\int_{\Om\setminus K}|\nabla u|^2\, dx + \hn(\Om\cap K),
\end{equation}
where $\hn$ is the $(N-1)$-dimensional Hausdorff measure and $(u,K)$ is any
pair such that $K$ is a closed subset of $\R^N$ and $u$ belongs to the 
Deny-Lions space $L^{1,2}(\Om\setmeno K)$ (we refer to Section~\ref{prel} for the definition of this space). 
In the sequel the class of all such pairs will be denoted by ${\mathcal A}(\Om)$ and its
elements will be called admissible pairs. 
The functional \eqref{MS}, which was introduced in \cite{MS1,MS2} in the context of image segmentation problems,
arises also in variational models for fracture mechanics (see \cite{Gr20} and \cite{FrMa98}). 

Let $(u,K)\in{\mathcal A}(\Om)$ be a {\em Dirichlet minimizer\/} of $F$, that is,
\begin{equation}\label{min0}
F(u,K)\leq F(v,K')
\end{equation}
for every $(v,K')\in{\mathcal A}(\Om)$ with $v=u$ on $\partial\Om$ in the sense of traces.
It is well known that $u$ is harmonic in $\Om\setmeno K$ and satisfies a
Neumann condition on $K$; more precisely, $u$ solves the equation
\begin{equation} \label{wHarm}
\begin{cases}
\Delta u =0 & \text{in } \Om\setmeno K, \\
\partial_\nu u=0 &\text{on } K.
\end{cases}
\end{equation}
As for the regularity of the discontinuity set $K$,
one can prove (see \cite{AFP} and \cite{Bon}) that
$K\cap\Om$ can be decomposed as 
\begin{equation}\label{decK0}
K\cap\Om=\Gamma_r\cup\Gamma_s,
\end{equation}
where $\Gamma_s$ is closed with $\hn(\Gamma_s)=0$
and $\Gamma_r$ is an orientable $(N-1)$-manifold of class $C^\infty$. 
Since $u$ is of class $C^\infty$ up to $\Gamma_r$ by \eqref{wHarm}, 
the traces $\nabla u^\pm$ of $\nabla u$ are well defined on both sides of $\Gamma_r$.
By considering variations of $\Gamma_r$ one can show (see \cite{MS2})
that the minimality \eqref{min0} implies also the following transmission condition:
\begin{equation}\label{trans}
|\nabla u^+|^2-|\nabla u^-|^2=H \qquad \text{on } \Gamma_r,
\end{equation}
where $H$ is the mean curvature of $\Gamma_r$.
We point out that by \eqref{wHarm} and \eqref{trans}
the function $u$ is subject to overdetermined boundary conditions on $\Gamma_r$.
Exploiting this observation it has been proved in \cite{KLM} that $\Gamma_r$ has in fact analytic
regularity.

Let now $(u,K)$ be a {\em critical point\/} of $F$, that is, 
a pair in ${\mathcal A}(\Om)$ satisfying \eqref{wHarm}, \eqref{decK0},
and \eqref{trans}. Due to the nonconvexity of $F$ one cannot expect
these conditions to be in general sufficient for minimality.
Nevertheless using a calibration method it is possible to prove that 
critical points are Dirichlet minimizers on small domains.
More precisely, it has been proved in \cite{MM01} that,
if $N=2$, for every regular arc $\Gamma$ compactly contained in $\Gamma_r$ there
exists a tubular neighbourhood $U$ of $\Gamma$ such that $(u,K)$ is a Dirichlet minimizer of $F$
in $U$.
The minimality on large domains can fail in a rather surprising way: there might exist critical points whose energy can be strictly
lowered by considering arbitrarily small diffeomorphic deformations of the regular part $\Gamma_r$ of $K$. 
An example of this phenomenon was given in \cite[Proposition~4.1]{MM01} by considering the critical point $(u_0,K_0)$, where
$u_0(x,y)=x$ for $y\geq0$, $u_0(x,y)=-x$ for $y<0$, and $K_0=\{y=0\}$. If $\Omega$ is the rectangle
$(x_0,x_0+\ell){\times}(-y_0,y_0)$ with $\ell$ and $y_0$ large enough, one can show that the functional can be decreased
by perturbing $\Gamma_r=(x_0,x_0+\ell){\times}\{0\}$ by a diffeomorphism arbitrarily close to the identity. 

In this paper we begin a study of second order necessary conditions for minimality.
More precisely, given a Dirichlet minimizer $(u,K)$, we compute the second derivative of the 
energy along variations of the form $(u_\e,K_\e)$,
where $K_\e=\Phi_\e(K\cap \Om)$, $\Phi_\e$ being a one-parameter family of diffeomorphisms
coinciding with the identity on a fixed neighbourhood of $\Gamma_s\cup\partial\Om$,
and $u_\e$ is the solution of the problem
$$
\min\left\{ \int_\Om |\nabla v|^2\, dx : \ v\in L^{1,2}(\Om\setmeno K_\e), \ v=u \text{ on } \partial\Om
\right\}.
$$
This approach has some similarities with the computation of ``shape derivatives" introduced in \cite{Zo}
in the context of shape optimization problems.  
As $\Phi_\e$ coincides with the identity on a neighbourhood of $\Gamma_s$, the singular part
of the discontinuity set is left unchanged by the variation, which can thus affect only
the regular part $\Gamma_r$. 
We also point out that the variation $u_\e$ of the function $u$ has a nonlocal character. 
This is crucial to retrieve information about global properties, 
such as the size and the geometry of $\Om$ and $\Gamma_r$.

Whereas the first order variation of $F$ along $(u_\e, K_\e)$ gives back
the equilibrium condition \eqref{trans}, the second order variation provides us
with a new necessary minimality condition, expressed in terms of a sign
condition for a quadratic form depending on $(u,K)$. More precisely, for every 
submanifold $\Gamma$ compactly contained in $\Gamma_r$ and every Lipschitz domain $U\subset\Omega$
we consider the functional $\delta^2\!F((u,\Gamma);U)$ on $H^1_0(\Gam)$
defined as
\begin{multline}\nonumber
\delta^2\!F((u,\Gamma);U)[\varphi]:= 2\int_{\Gam} (\vf^+\partial_\nu\vf^+ - \vf^-\partial_\nu\vf^- )\, d\hn
+\int_{\Gam} |\nablat \varphi |^2\, d\hn 
\\
{}+\int_{\Gam} (2\B[\nablat u^+,\nablat u^+]
-2\B[\nablat u^-,\nablat u^-] -|\B|^2) \varphi^2\, d\hn
\end{multline}
for every $\varphi\in H^1_0(\Gam)$, where $\vf\in L^{1,2}(U\setmeno K)$ solves the problem 
$$
\begin{cases}
\Delta \vf =0 & \text{in } U\setmeno K, \\
\vf=0 & \text{on }\partial U, \\
\partial_\nu \vf^\pm=\divt (\varphi\nablat u^\pm) &\text{on } \Gam, \\
\partial_\nu \vf=0 & \text{on } K\cap U\setmeno\Gamma.
\end{cases}
$$
Here the symbols $\nabla_\G$ and $\div_\G$ denote the tangential gradient and the
tangential divergence on $\G$, $\B$ is the second fundamental form of $\G$, while
$\vf^\pm$ denote the traces of $\vf$ on the two sides of $\Gamma$.
As $\vf$ depends linearly on $\varphi$, the functional 
$\delta^2\!F ((u,\Gamma);U)$ defines a quadratic form on $H^1_0(\Gam)$.

We first show (Theorem~\ref{th:necessary}) that, if $(u,K)$ is a Dirichlet minimizer of $F$, 
then for every $\Gamma$ and $U$ as above
we have the second order condition
\begin{equation}\label{into2}
\delta^2\!F((u,\Gamma);U)[\varphi]\geq0 \qquad \text{for every } \varphi\in H^1_0(\Gam).
\end{equation}
Conversely, we prove (Theorem~\ref{thm:suff}) that, if $N\leq3$ and $(u,K)$ is a critical point satisfying the stronger condition
\begin{equation}\label{into1}
\delta^2\!F((u,\Gamma);U)[\varphi]>0 \qquad \text{for every } \varphi\in H^1_0(\Gam)\setmeno\{0\},
\end{equation}
then $(u,K)$ is a minimizer of $F$ on $U$
with respect to all pairs $(v,\Phi(K\cap U))$ such that $\Phi$ is a diffeomorphism belonging to a $C^2$-neighbourhood
of the identity and coinciding with the identity on $K\cap U\setmeno\G$, and $v\in L^{1,2}(U\setmeno \Phi(K\cap U))$
with $v=u$ on $\partial U$.
The restriction $N\leq3$ is a technical assumption. In fact, a slightly weaker minimality property
is shown to hold in any dimension (see Remark~\ref{rmk:add}).

A detailed study of the stronger condition \eqref{into1}
is performed in Section~\ref{sec:eq}, where two equivalent formulations are shown.
The first one (Theorem~\ref{thm:eq1}) is a condition on the first eigenvalue of the (nonlocal) compact operator $T:H^1_0(\Gam)\to H^1_0(\Gam)$,
defined for every $\varphi \in H^1_0(\Gam)$ as
$$
T\varphi= R\big(2\nablat u^+{\,\cdot\,}\nablat \vf^+
-2\nablat u^-{\,\cdot\,}\nablat \vf^-\big).
$$
Here $R:H^{-1}(\Gam)\to H^1_0(\Gam)$ denotes the resolvent operator which maps
$f\in H^{-1}(\Gam)$ into the solution $\theta$ of the problem
$$
\begin{cases}
-\Deltat \theta +a\theta=f &
\text{in }\Gamma\cap U,  \\
\theta \in H^1_0(\Gam),
\end{cases}
$$
where $\Deltat$ is the Laplace-Beltrami operator on $\Gamma$
and 
$$
a(x)=2 \B(x)[\nablat u^+(x),\nablat u^+(x)]-2 \B(x)[\nablat u^-(x),\nablat u^-(x)] -|\B(x)|^2.
$$
The second equivalent formulation (Theorem~\ref{thm:eq2}) is expressed in terms of the variational problem
$$
\min\Big\{2 \int_U |\nabla v|^2\, dx : \ v\in L^{1,2}(U\setmeno K), \ v=0 \text{ on } \partial U,
\ \int_{\Gam} (a\psi_v^2+|\nablat\psi_v|^2)\, d\hn=1\Big\},
$$
where $\psi_v=R(2\nablat u^+{\,\cdot\,}\nablat v^+ -2\nablat u^-{\,\cdot\,}\nablat v^-)$. This minimum problem
describes a sort of nonlocal ``capacity" of $\Gamma$ with respect to $U$, where the usual pointwise constraint
$v=1$ a.e.\ on $\Gam$ is replaced by the integral condition on $\psi_v$.
We also note that this second formulation is strictly related to the sufficient condition for graph-minimality studied in \cite{MM01}.
It is easy to see that the sufficient condition in \cite{MM01} is stronger than \eqref{into1}
and in fact it implies a stronger minimality property. The comparison between the two conditions is discussed in the
explicit example of Section~\ref{sec:ex} (see Remark~\ref{rmk:ex}), where we consider the critical point $(u_0,K_0)$
of \cite[Proposition~4.1]{MM01} and we give a complete characterization 
of the rectangles $U=(x_0,x_0+\ell){\times}(-y_0,y_0)$
where condition \eqref{into1} is satisfied.

Finally, we prove some stability and instability results. We first show that,
if $(u,K)$ is a critical point, then condition \eqref{into1} is automatically satisfied 
when the domain or the support of the variation is sufficiently small (Propositions~\ref{thm:small} 
and~\ref{thm:smsupp}). Instead, condition \eqref{into2} may fail if the domain is too large (Proposition~\ref{thm:fail}).
This is in agreement with the two dimensional results of \cite{MM01}. 

It remains an open problem to understand whether condition \eqref{into1} implies a stronger minimality property, 
in analogy to the classical results of the Calculus of Variations for weak minimizers.
This would probably require the use of different techniques,
such as calibration methods or Weierstrass fields theory.

It is our intention to investigate variations involving also the
singular part $\Gamma_s$ of the discontinuity set in future work.
Moreover, it is our belief that the techniques developed in this paper 
can be applied to more general functionals, both in the bulk and in the surface energy.

The plan of the paper is the following. In Section~\ref{prel} we collect all the notation
and the preliminary results needed in the paper. Section~\ref{sec:var2} is devoted to
the derivation of the second order necessary condition \eqref{into2}. In Section~\ref{sec:eq} we discuss the
equivalent formulations of the sufficient condition \eqref{into1}, which is proved in Section~\ref{sec:suf}.
Stability and instability results are the subject of Section~\ref{sec:st}, while the explicit example
in dimension $2$ is studied in Section~\ref{sec:ex}. Finally, the regularity results needed in the derivation argument
are collected and proved in Section~\ref{appendix}.
\end{section}

\begin{section}{Notation and preliminaries}\label{prel}

In this section we fix the notation and we recall some preliminary results.
\medskip

\noindent
{\bf Matrices and linear operators.} 
Given a linear operator $A:\R^N\to \R^d$, we denote the action of $A$ 
on the generic vector $h\in \R^N$ by $A[h]$. 
We will usually identify linear operators with matrices. 
We denote the euclidean norm of a linear operator (or a matrix) $A$ by
$$
|A|:=({\rm trace}\,(A^TA))^{1/2},
$$
where $A^T:\R^d\to \R^N$ stands for the adjoint operator.  
If $d=N$ we can consider the bilinear form associated with $A$
$$
A[h_1, h_2]:=A[h_1]{\,\cdot\,} h_2\qquad\text{for }h_1, h_2\in \R^N,
$$
where the dot denotes the scalar product of $\R^N$. 
Conversely, to any bilinear form $B:\R^N{\times}\R^N\to\R$ 
we can naturally associate a linear operator, still denoted by $B$, 
whose action on the generic vector $h\in\R^N$ can be described by duality as
$$
B[h]{\,\cdot\,} z=B[h,z]\qquad\text{for every }z\in\R^N.
$$
We will usually identify bilinear forms with the associated linear operators.
\medskip

\noindent 
{\bf Geometric preliminaries.}
Let $\Gamma\subset\R^N$ be a smooth orientable $(N-1)$-dimensional manifold and 
assume that there exists a smooth orientable $(N-1)$-dimensional manifold $M$ 
such that $\Gamma\subset\subset M$.
For every $x\in \Gamma$ we denote the tangent space and the normal space 
to $\Gamma$ at $x$ by $T_x\Gamma$ and $N_x\Gamma$, respectively. 

Let $S^{N-1}$ be the $(N-1)$-dimensional unit sphere in $\R^N$.
We call an {\em orientation} for $\Gamma$ any smooth vector field
$\nu:\Gamma\to S^{N-1}$ such that $\nu(x)\in N_x\G$ for every $x\in\G$.
Given an orientation we can define a signed distance function from $\G$, which turns out to be
smooth in a tubular neighbourhood ${\mathcal U}$ of $\G$ and whose gradient coincides with $\nu$ on $\G$.
The extension of the normal vector field 
provided by the gradient of the signed distance function will be still denoted by
$\nu:{\mathcal U}\to S^{N-1}$.

We now recall the definition of some tangential differential operators.  
Let $g:\mathcal{U}\to \R^d$ be a smooth function.
The {\em tangential differential} $d_\Gamma g(x)$ of $g$ at $x\in\Gamma$ is the linear operator
from $\R^N$ into $\R^d$ given by $d_\Gamma g(x):=dg(x)\circ \pi_x$,
where $dg(x)$ is the usual differential of $g$ at $x$
and $\pi_x$ is the orthogonal projection on $T_x\Gamma$. 
We denote the matrix (the vector if $d=1$) associated
with $d_\Gamma g(x)$ by $D_\Gamma g(x)$ ($\nabla_\Gamma g(x)$ if $d=1$).
As remarked above we will often identify matrices with linear operators.   
Note that
$$
(D_\Gamma g(x))^T[h]{\,\cdot\,}\nu(x)=h{\,\cdot\,}D_\Gamma g(x)[\nu(x)]=0\quad \text{for every }
h\in\R^d,
$$
that is, $(D_{\Gamma}g(x))^T$ maps $\R^d$ into $T_x\Gamma$. 
We remark also that by our choice of the extension of $\nu$ around $\Gamma$
we have
\begin{equation}\label{NU}
D\nu=D_\Gamma \nu \qquad\text{on } \G.
\end{equation}
If $d=N$ we can define the {\em tangential divergence} $\div_{\Gamma} g$ of $g$ as 
$$
\div_{\Gamma} g:=\sum_{j=1}^Ne_j{\,\cdot\,}\nabla_{\Gamma}g_j,
$$
where $e_1,\dots, e_N$ are the vectors of the canonical basis of $\R^N$ and $g_1,\dots, g_N$ are 
the corresponding components of $g$. It turns out that 
$$
\div_{\Gamma} g(x)=\sum_{j=1}^{N-1}\tau_j{\,\cdot\,}\partial_{\tau_j}g,
$$
where $\tau_1,\dots, \tau_{N-1}$ is any orthonormal basis of $T_x\Gamma$ and
for every $v\in S^{N-1}$ the symbol $\partial_v$ denotes the derivative in the direction $v$.
Sometimes it is also useful to bear in mind the identity 
$$
\div\, g=\div_{\Gamma}g+\nu{\,\cdot\,}\partial_{\nu}g.
$$ 
In particular, as $\partial_{\nu}\nu=0$ by \eqref{NU},
we deduce that
$\div\,\nu=\div_{\Gamma}\nu$ on $\Gamma$. We will make repeated use of the following identities:
$$
\begin{array}{c}
 \div_{\Gamma} (\varphi g_1) 
=  \nabla_{\Gamma} \varphi {\,\cdot\,} g_1 
+ \varphi \, \div_{\Gamma} g_1, 
\medskip\\
 \nabla_{\Gamma} (g_1{\,\cdot\,} g_2)
= ( D_{\Gamma} g_1 )^T [g_2] 
+ ( D_{\Gamma} g_2 )^T [g_1].
\end{array}
$$
for $\varphi \in C^1(\mathcal{U})$ and $g_1,g_2 \in C^1(\mathcal{U};\R^N)$. 
Finally, we recall that the {\em Laplace-Beltrami operator $\Delta_\G$ on
$\Gamma$} is defined as
$$
\Delta_{\Gamma}g:=\div_{\Gamma}(\nabla_{\Gamma}g)
$$
for every smooth real valued function $g$. We remark that all the tangential differential operators
introduced so far have an intrisic meaning, since they only depend on the restriction
of $g$ to $\Gamma$. 

For every $x\in\Gamma$ we set 
\begin{equation}\label{sff}
\B(x):=D_{\Gamma}\nu(x)=D\nu(x).
\end{equation}
The bilinear form associated with $\B(x)$ is symmetric and, when restricted to 
$T_x \Gamma{\times}T_x \Gamma$, it coincides with the {\em second fundamental form
of $\Gamma$ at $x$}. 
It is also possible to prove that $T_x\Gamma$ is an invariant space for $\B(x)$. 
 
We consider also the function $H:\mathcal{U}\to \R$ defined by
\begin{equation}\label{defH}
H:=\div\, \nu.
\end{equation}
On $\Gamma$ we have $H=\div\, \nu=\div_{\Gamma} \nu={\rm trace\,}\B$,
that is, for every $x\in\Gamma$ the value
$H(x)$ coincides with the {\em mean curvature of $\Gamma$ at $x$}. 

It is important to recall the following {\em divergence formula}:
\begin{equation}\label{fdiv}
\int_{\G} \div_{\G} g\,d \hn = \int_{\G} H  (g{\,\cdot\,}\nu)\, d\hn,
\end{equation}
which holds for every smooth function $g:\mathcal{U}\to \R^N$ with
$\supp g\cap\Gamma\subset\subset\Gamma$.
Note that \eqref{fdiv} allows to extend to tangential operators the usual integration by parts formula. 
Indeed, we have
\begin{equation}\label{fparts}
\int_{\G}\varphi\, \div_{\G} g\,d \hn =
- \int_{\G} \nabla_{\Gamma} \varphi {\,\cdot\,} g\, d\hn
\end{equation}
for every smooth $g:\mathcal{U}\to \R^N$ such that $g(x)\in T_x\Gamma$ for $x\in\Gamma$, and 
every smooth $\varphi:\mathcal{U}\to \R$ with  
$\supp \varphi\cap\Gamma\subset\subset\Gamma$.

Let $U$ be a bounded open set in $\R^N$ with $U\cap\Gamma\neq\varnothing$ and
let $\Phi:\overline{U}\to\overline{U}$ be a smooth orientation-preserving diffeomorphism.
Then $\Gamma_{\Phi}:=\Phi(\Gamma\cap U)$ is still an orientable smooth $(N-1)$-manifold.
A possible choice for the orientation is given by the vector field
\begin{equation}\label{nutfit0}
\nu_\Phi
= \frac{(D\Phi)^{-T}  [\nu]}{|(D\Phi )^{-T} [\nu]|}\circ\Phi^{-1}.
\end{equation}
Accordingly we can define the functions $\B_{\Phi}$ and $H_{\Phi}$ as in \eqref{sff} and \eqref{defH}, 
with $\Gamma$ and $\nu$ replaced by $\Gamma_{\Phi}$ and $\nu_{\Phi}$, respectively.
We shall use  the following identity, which is  a particular case of the so-called generalized area formula 
(see, e.g., \cite[Theorem~2.91]{AFP}): 
for every $\psi\in L^1(\Gamma_\Phi)$
\begin{equation}\label{farea}
\int_{\Gamma_\Phi} \psi \,d\hn =\int_\Gamma (\psi\circ\Phi) J_\Phi \, d\hn,
\end{equation}
where $J_\Phi:=|(D\Phi)^{-T}[\nu]|\det D\Phi$ is the $(N-1)$-dimensional Jacobian of $\Phi$.

We conclude this subsection by introducing the Sobolev space $H^1_0(\Gamma)$, which is defined as the closure 
of $C^{\infty}_c(\Gamma)$ with respect to the norm
$$
\|u\|^2_{H^1(\Gamma)}:=\int_\Gamma \big(|u|^2+|\nabla_{\Gamma}u|^2\big)\, d\hn.
$$ 
Many of the properties of classical Sobolev spaces, such as Poincar\'e inequalities and integration
by parts formulas continue to hold. We refer to \cite{Hebey} for a complete treatment of
these spaces.  We shall denote the dual space of $H^1_0(\Gamma)$ by $H^{-1}(\Gamma)$.
\medskip

\noindent 
{\bf Deny-Lions spaces.} 
Given a bounded open subset $\Om \subset \R^N$, we say that $\Om$
has a Lipschitz boundary at a point $x \in \partial \Om$ if 
there exist an orthogonal coordinate system $(y_1, \dots, y_N)$, a
coordinate rectangle $R = (a_1, b_1){\times}\dots{\times}(a_N, b_N)$ containing $x$, 
and a Lipschitz function $\Psi:(a_1, b_1){\times}\dots{\times}(a_{N-1}, b_{N-1})\to (a_N, b_N)$ such that $\Om \cap R = 
\{y \in R :\, y_N < \Psi(y_1, \dots, y_{N-1}) \}$.
The set of all such points, which is by definition relatively open, is denoted by 
$\partial_L \Om$. If $\partial_L \Omega = \partial \Om$ we say that $\Om$ is a Lipschitz domain.

To deal with possibly unbounded functions in problem \eqref{wHarm}, besides the classical Sobolev
space $H^1(\Om)$ we shall also use the {\em Deny-Lions space}
$$
L^{1,2}(\Om):= \{ u \in L^2_{loc}(\Om):
\, \nabla u \in L^2 (\Om; \R^N)\},
$$
which coincides with the space of all distributions on $\Om$ whose gradient belongs to
$L^2(\Om; \R^N)$.
In the brief account below we essentially follow \cite[Section~2]{DT02} (see also \cite{DL54}). 
The relation between Sobolev and Deny-Lions spaces is unveiled by the following proposition.

\begin{proposition} \label{pr2}
Let $u \in L^{1,2}(\Om)$ and let $x \in \partial_L \Om$. Then there exists a neighbourhood
$U$ of $x$ such that $u|_{\Om \cap U}\in H^1 (\Om \cap U)$. 
In particular, if $\Om$ is Lipschitz, then 
$L^{1,2}(\Om)=H^1(\Om)$.
\end{proposition}

Let $A$ and $B$ be $\hn$-measurable sets in $\R^N$.
We say that $A$ is {\em quasi-contained} in $B$, and we write 
$A\capset B$, if $\hn(B\setmeno A)=0$. 
It is known that every function in $L^{1,2}(\Om)$ can be specified at 
$\hn$-a.e.\ point of $\Om\cup \partial_L\Om$.  
Hence, if $\Lambda\subset\partial\Om$ is relatively open and 
$\Lambda\capset \partial_L\Om$, we can define
the space
\begin{equation}\label{L120}
L^{1,2}_0(\Om; \Lambda):=\{u \in L^{1,2}(\Om) :\, 
u=0 \quad \hn\text{-a.e.\ on } \Lambda \},
\end{equation}
where we identify functions which differ by a constant on the connected components 
of $\Om$ whose boundary does not meet $\Lambda$.
With this identification, arguing as in \cite[Corollary 2.3]{DT02}, one can prove
the following.

\begin{proposition} \label{cor1}
The space $L^{1,2}_0(\Om;\Lambda)$ introduced in \eqref{L120}  
is a Hilbert space endowed with the norm
$\| \nabla u \|_{L^2(\Om;\R^N)}$.
\end{proposition}

\end{section}

\begin{section}{The second variation}\label{sec:var2}

In this section we define and compute a suitable notion of second variation 
for the Mumford-Shah functional \eqref{MS}.
We recall that $\A(\Om)$ is the class of all pairs $(u,K)$ such that $K$ 
is a closed subset of $\R^N$ and $u\in L^{1,2}(\Om\setmeno K)$.
It is useful to ``localize'' the definition of $F$ to any open
subset $U\subset\Om$ by setting
$$
F((u,K); U):=\int_{U\setminus K}|\nabla u|^2\, dx + \hn(U\cap K)
$$
for every admissible pair $(u,K)\in \A(\Om)$.  

In the sequel we shall consider only admissible pairs $(u,K)$ 
which are partially regular in the sense of the following definition. 

\begin{definition}\label{def:partreg}
Let $\Om\subset\R^N$ be a Lipschitz domain and let $(u,K)\in \A(\Om)$. 
We say that $(u,K)$ is {\em partially regular in $\Om$} if
$\partial\Om\capset \partial_L(\Om\setmeno K)$ 
(see the end of Section~\ref{prel}),
$u$ solves the problem
\begin{equation}\label{usolves}
\min\Big\{\int_{\Om\setminus K}|\nabla v|^2\, dx:\
v-u\in L^{1,2}_0(\Om\setmeno K; \partial \Om)\Big\},
\end{equation}
and $K$ can be decomposed as 
$K=\Gamma_r\cup \Gamma_s$, with $\Gamma_r\cap \Gamma_s=\varnothing$,
$\Gamma_s$ relatively closed, $\hn(\Gamma_s)=0$, and $\Gamma_r$
orientable $(N-1)$-manifold of class $C^{\infty}$. 
We denote the class of all such pairs by $\Ar(\Om)$. Finally we say that $U\subset \Om$ is 
an {\em admissible subdomain for $(u,K)$} if it is Lipschitz and $(u, K\cap U)\in \Ar(U)$.
\end{definition}

\begin{remark}
The previous definition is motivated by the regularity results 
for local minimizers of free discontinuity problems (see \cite{AFP} and 
\cite{Bon}).
Note also that $u$ solves \eqref{usolves} if and only if 
\begin{equation}\label{uequiv}
\int_{\Om}\nabla u{\,\cdot\,}\nabla z\, dx=0\qquad
\text{for every }z\in L^{1,2}_0(\Om\setmeno K; \partial \Om),
\end{equation}
where we also used the fact that $K$ has Lebesgue measure equal to 0.
\end{remark}
 
In the next definition we introduce the class of admissible variations 
of the discontinuity set $K$. 
 
\begin{definition}\label{flux}
Let $\Om\subset\R^N$ be a Lipschitz domain, let $(u,K)\in \Ar(\Om)$, 
let $\Gamma\subset\subset\Gamma_r$ be relatively open,
and let $U\subset \Om$ be an admissible subdomain for $(u,K)$ according to Definition~\ref{def:partreg}.
We say that $(\Phi_t)_{t\in(-1,1)}$ is an
{\em admissible flow for $\Gamma$ in $U$}
if the following properties are satisfied:
\begin{itemize}
\item[(i)] the map $(t,x)\mapsto \Phi_t(x)$ belongs to $C^{\infty}((-1,1){\times}\overline U; \overline U)$;
\item[(ii)] for every $t\in(-1,1)$ the map $\Phi_t$ is a diffeomorphism from $\overline U$ onto itself;
\item[(iii)] $\Phi_0$ coincides with the identity map $I$ in $\overline U$;
\item[(iv)] there exists a compact set $G\subset U\setminus(K\setminus \Gamma)$ such 
that $\supp(\Phi_t-I)\subset G$  for every $t\in(-1,1)$.
\end{itemize}
\end{definition}

\begin{remark}
Condition (iv) in the previous definition implies that 
$\Phi_t$ can affect $\Gamma$ only, while   
$(K\cap U)\setmeno \Gamma$ remains unchanged. We also remark that
from the assumptions 
$\Gamma$ has positive distance from $\Gamma_s\cup \partial\Om$, where 
singular behaviour of the function $u$ can occur. 
\end{remark} 
 
Finally, we describe the variation of $u$ associated with an 
admissible variation of its discontinuity set $K$.
Let $\Om$, $(u,K)$, and $U$ be as in Definition~\ref{flux}.
Given a diffeomorphism $\Phi\in C^{\infty}(\overline U; \overline U)$,
satisfying condition (iv) (with $\Phi_t$ replaced by $\Phi$), we define
$u_{\Phi}$ as the (unique) solution of 
\begin{equation}\label{uphi}
\begin{cases}
u_{\Phi}-\tilde u\in  L^{1,2}_0(U\setmeno K_{\Phi}; \partial U),
\smallskip\\
\displaystyle 
\int_{U}\nabla u_{\Phi}{\,\cdot\,}\nabla z\, dx=0\quad 
\text{for every }z\in L^{1,2}_0(U\setmeno K_{\Phi}; \partial U),
\end{cases}
\end{equation}
where $\tilde u:=\tilde\varphi\, u$ and $\tilde \varphi$ is
a cut-off function such that $\tilde\varphi=0$ on $G$ and $\tilde\varphi=1$ 
in a neighbourhood of $\partial U$. In particular, $u_\Phi=u$ $\hn$-a.e.\ on $\partial U$.

We are now ready to define our notion of second variation.
 
\begin{definition}\label{def:varsec}
Let $\Om$, $(u,K)$, $U$, $\Gamma$, and $(\Phi_t)$ be as in Definition~\ref{flux}. 
We define the {\em second variation of $F$ at $(u, K)$ in $U$ along the flow $(\Phi_t)$} to be the value of 
\begin{equation}\label{eq:varsec}
\frac{d^2}{dt^2}F((u_{\Phi_t}, K_{\Phi_t});U)|_{t=0},
\end{equation}
where $K_{\Phi_t}:=\Phi_t(K\cap U)$ and $u_{\Phi_t}$ is defined by \eqref{uphi} with $\Phi$ replaced by $\Phi_t$. 
\end{definition}

We point out that the existence of the derivative \eqref{eq:varsec}
is guaranteed by the regularity results of Section~\ref{appendix}.

We fix now some notation which will be repeatedly used in the following
discussion. 
For any one-parameter family of function $(g_s)_{s\in(-1,1)}$ the symbol
$\dot g_t(x)$ will denote the partial derivative with respect to $s$
of the map $(s,x)\mapsto g_s(x)$ evaluated at $(t,x)$.
To be more specific, let $\Om$, $(u,K)$, $U$, $\Gamma$, and $(\Phi_t)$ be as in 
the previous definition. For every $t\in (-1,1)$ we set
$$
X_{\Phi_t}:=\dot{\Phi}_t\circ\Phi_t^{-1}, \qquad Z_{\Phi_t}:=\ddot{\Phi}_t\circ\Phi_t^{-1},
$$
where, according to the previous notation,
$$
\dot{\Phi}_t:=\frac{\partial}{\partial s}{\Phi}_s|_{s=t},
\qquad
\ddot{\Phi}_t:=\frac{\partial^2}{\partial s^2}{\Phi}_s|_{s=t}.
$$
Similarly, for every $t\in (-1,1)$ we define $\dot u_{\Phi_t}$ as
the partial derivative with respect to $s$ of the map $(s,x)\mapsto
u_{\Phi_s}(x)$ evaluated at $(t,x)$. 
Proposition~\ref{prop:reg} in the appendix guarantees that the derivative exists and
that $\dot u_{\Phi_t}\in L^{1,2}_0(U\setmeno K_{\Phi_t}; \partial U)$.
We shall often omit the subscript when $t=0$; in particular, we set
\begin{equation}\label{noT}
\dot u:=\dot u_{\Phi_0},\qquad X:=X_{\Phi_0}, \qquad Z:=Z_{\Phi_0}.
\end{equation}
We define $X^{\parallel}$ as the orthogonal projection of $X$ onto the tangent space
to $\Gamma$, that is, $X^{\parallel}:=(I-\nu\otimes\nu)X$.
Finally, for any function $z\in L^{1,2}(\Om\setmeno K)$ we denote 
the traces of $z$ on the two sides of $\Gamma$ by $z^+$ and $z^-$. More precisely, for $\hn$-a.e.\ 
$x\in\Gamma$ we set 
$$
z^{\pm}(x):=\lim_{r\to 0}\tfrac{1}{\LL^N(B_r(x)\cap V^{\pm}_x)}
\int_{B_r(x)\cap V^{\pm}_x} z(y)\, dy,
$$
where $\LL^N$ is the $N$-dimensional Lebesgue measure, $B_r(x)$ is the open ball of radius $r$ centered at $x$, and $V^{\pm}_x:=\{y\in\R^N:\ \pm (y-x){\,\cdot\,}\nu(x)\geq 0\}$.

In the next theorem, which is the main result of the section,
we compute the second variation of $F$, according to Definition~\ref{def:varsec}.
We refer to Section~\ref{prel} for the definition of all geometrical quantities appearing in the statement. 
 
\begin{theorem}\label{th:var2}
Let $\Om$, $(u,K)$, $U$, $\Gamma$, and $(\Phi_t)$ be as in Definition~\ref{flux}. Then
the function $\dot{u}$ belongs to $L^{1,2}_0(U\setmeno K; \partial U)$ and satisfies
the equation
\begin{equation} \label{upto-solves}
\int_U \nabla \dot{u}{\,\cdot\,} \nabla z\, dx+\int_{\Gam}\big[\div_{\Gamma} 
( (X {\,\cdot\,} \nu) \nabla_{\Gamma} u^+  )z^+-\div_{\Gamma} 
( (X {\,\cdot\,} \nu) \nabla_{\Gamma} u^-  )z^-\big]\, d\hn=0
\end{equation}
for all $z\in L^{1,2}_0(U\setmeno K; \partial U)$.
Moreover, 
the second variation of $F$ at $(u, K)$ in $U$ along the flow $(\Phi_t)$ is given by
\begin{eqnarray}
& \displaystyle
\frac{d^2}{dt^2}F((u_{\Phi_t}, K_{\Phi_t});U)|_{t=0}
= 2 \int_{\Gam} ( \dot{u}^+
\partial_{\nu} \dot{u}^+ 
- \dot{u}^- \partial_{\nu} \dot{u}^- ) \, d\hn
+ \int_{\Gam} |\nabla_{\G} 
( X{\,\cdot\,} \nu ) |^2\, d\hn
\nonumber \\
& \displaystyle \hphantom{\int_\G(X {\,\cdot\,} \nu)^2}
{}+\int_{\Gam} ( X {\,\cdot\,} \nu)^2 
(  2 \B[\nabla_{\Gamma} u^+, \nabla_{\Gamma} u^+] 
- 2 \B[\nabla_{\Gamma} u^-, \nabla_{\Gamma} u^-]
- |\B |^2)\,  d\hn \label{fvar2}
\\
& \displaystyle \hphantom{\int_\G(X {\,\cdot\,} \nu)^2}
{}+\int_{\Gam}  f ( Z {\,\cdot\,} \nu 
-2 X^{\parallel} {\,\cdot\,} 
\nabla_{\Gamma} ( X {\,\cdot\,} \nu )
+\B[X^{\parallel}, X^{\parallel}]
+H( X {\,\cdot\,} \nu)^2)\, d\hn, \nonumber
\end{eqnarray}
where $f:=| \nabla_{\Gamma} u^- |^2 
- | \nabla_{\Gamma} u^+ |^2 + H$.
\end{theorem} 

\begin{remark}\label{rm:upto}
The first part of the previous theorem implies that
$\dot{u}$ is harmonic in $U\setmeno K$, $\dot{u}=0$ on $\partial U$,
$\partial_{\nu}\dot{u}^\pm=\div_{\Gamma}((X {\,\cdot\,} \nu) \nabla_{\Gamma} u^\pm)$
on $\Gam$, and $\dot{u}$ satisfies a weak homogeneous Neumann condition on $K\cap U\setmeno\Gamma$.
In particular, using $\dot u$ as a test function in \eqref{upto-solves}, we have 
$$
\int_{\Gam} ( \dot{u}^+\partial_{\nu} \dot{u}^+ 
- \dot{u}^- \partial_{\nu} \dot{u}^- ) \, d\hn
=-\int_U|\nabla \dot u|^2\, dx.
$$
\end{remark}

The following lemma contains some useful identities, which will be repeatedly used in the proof
of Theorem~\ref{th:var2}. The proof of the lemma is postponed until Section~\ref{appendix}.

\begin{lemma}\label{lem:identita}
The following identities are satisfied on $\Gamma$:
\smallskip
\begin{itemize}
\item[(a)] $\nabla^2 u^\pm [ \nu , \nu ] = 
- \Delta_{\Gamma} u^\pm$;
\medskip
\item[(b)] $\nabla^2 u^\pm [X, \nu] = 
-( X {\,\cdot\,} \nu ) \Delta_{\Gamma} u^\pm
-\B[ \nabla_{\Gamma} u^\pm , X ]$;
\medskip
\item[(c)] $\div_{\Gamma}[ ( X {\,\cdot\,} \nu ) \nabla_{\Gamma} u^\pm] =
( D_{\Gamma} X )^T [ \nu , \nabla_{\Gamma} u^\pm ]
- \nabla^2 u^\pm [ X , \nu ]$;
\medskip
\item[(d)] $\partial_{\nu}H = -|\B|^2$;
\medskip
\item[(e)] $\nabla^2 u^{\pm} [ \nu, \nabla_{\Gamma} u^{\pm} ] 
= -\B[ \nabla_{\Gamma} u^{\pm} , \nabla_{\Gamma} u^{\pm}]$; 
\medskip
\item[(f)] $\dot{\nu} =  -(D_{\Gamma}X)^T [ \nu] -D_\Gamma\nu[ X ]$;
\smallskip
\item [(g)] $\displaystyle \frac{\partial}{\partial t}\big( 
\dot{\Phi}_{t} {\,\cdot\,} ( \nu_{\Phi_{t}} \circ \Phi_{t} ) \, J_{\Phi_{t}}  
\big)|_{t=0}
= Z {\,\cdot\,} \nu 
-2 X^{\parallel} {\,\cdot\,} 
\nabla_{\Gamma} ( X{\,\cdot\,} \nu ) 
+ \B[X^{\parallel}, X^{\parallel}] 
+ \div_{\Gamma} \big( (X {\,\cdot\,} \nu) X \big)$.
\end{itemize}
\smallskip
\end{lemma}

We will also need the following well-known result on the first variation of the area
functional (for the definition of $\G_{\Phi_{t}}$ and $H_{\Phi_{t}}$ we refer to Section~\ref{prel}). 

\begin{proposition}[see \cite{Sim}]
The first variation for the area functional is given by
\begin{equation} \label{1varareat}
\frac{d}{dt} \mathcal{H}^{N-1}(\G_{\Phi_{t}})  =
\int_{\G_{\Phi_{t}}} H_{\Phi_{t}} ( X_{\Phi_{t}} {\,\cdot\,} \nu_{\Phi_{t}} ) \, d \mathcal{H}^{N-1}.
\end{equation}
\end{proposition}

We are now in a position to prove Theorem~\ref{th:var2}.

\begin{proof}[Proof of Theorem~\ref{th:var2}]
We split the proof into three steps.
\medskip

\noindent
{\it Step 1. Derivation of the equation solved by $\dot u$.\/}
By Proposition~\ref{prop:reg} we have that 
$\dot u\in L^{1,2}_0(U\setmeno K; \partial U)$.
Let $z\in L^{1,2}_0(U\setmeno K; \partial U)$ with $\supp z\cap \overline \Gamma=\varnothing$. 
Then, $\supp z\subset U\setmeno K_{\Phi_{t}}$ for $t$ small enough, so that, in particular, 
$z\in L^{1,2}_0(U\setmeno K_{\Phi_{t}}; \partial U)$. Hence, by \eqref{uphi} we have 
$\int_U\nabla u_{\Phi_{t}}{\,\cdot\,}\nabla z\, dx=0$. Differentiating with respect to $t$, we deduce  
\begin{equation}\label{z2}
\int_U\nabla \dot{u}{\,\cdot\,}\nabla z\, dx=0\qquad\text{for every }
z\in L^{1,2}_0(U\setmeno K; \partial U) \text{ with }\supp z\cap \overline \Gamma=\varnothing.
\end{equation}

Note that by \eqref{nutfit0} one has
\begin{equation}\label{nutfit}
\nu_{\Phi_{t}} \circ\Phi_{t} 
= \frac{(D\Phi_{t} )^{-T}  [\nu]}{|(D\Phi_{t} )^{-T} [\nu]|} \quad\text{on }\Gamma.
\end{equation}
It is convenient to set $w_{t}:=(D\Phi_{t} )^{-T}  [\nu]$ (as usual, we shall omit the subscript $t$
when $t=0$). 
As $\partial_{\nu_{\Phi_{t}}}u^\pm_{\Phi_{t}}=0$ on $\Gamma_{\Phi_{t}}$ by \eqref{uphi}, 
we have $(\nabla u^\pm_{\Phi_{t}} \circ\Phi_{t} ){\,\cdot\,}(\nu_{\Phi_{t}} \circ\Phi_{t} )=0$
and in turn, using \eqref{nutfit},   
\begin{equation}\label{inturn}
(\nabla u^\pm_{\Phi_{t}} \circ\Phi_{t} ){\,\cdot\,} w_{t}=0 \qquad\text{on }\Gamma.
\end{equation}
Differentiating \eqref{inturn} with respect to $t$ at $t=0$ and using the fact that
$\dot w=-(DX)^T[\nu]$ on $\Gamma$, we obtain  
$$
\partial_{\nu}\dot{u}^\pm=-\nabla^2u^\pm[X,\nu]+(DX)^T[\nu,\nabla u^\pm]=
-\nabla^2u^\pm[X,\nu]+(D_{\Gamma}X)^T[\nu,\nabla_{\Gamma} u^\pm],
$$
where in the last equality we used that $\nabla u^\pm=\nabla_{\Gamma}u^\pm$ on $\Gamma$. 
By (c) of Lemma~\ref{lem:identita} we conclude that
\begin{equation}\label{z1}
\partial_{\nu} \dot{u}^\pm =\div_{\Gamma} [ (X {\,\cdot\,} \nu) \nabla_{\Gamma}u^\pm ]\qquad
\text{on $\Gamma$.}
\end{equation}

Now let $z\in L^{1,2}_0(U\setmeno K; \partial U)$. We can write $z=z_1+z_2$, where 
$\supp z_1\subset\subset U$ and $\supp z_1\cap \Gamma_s=\varnothing$, while $\supp z_2\cap 
\overline \Gamma=\varnothing$. Then, by \eqref{z2} and \eqref{z1} we finally obtain
$$
\int_U \nabla \dot{u}{\,\cdot\,} \nabla z\, dx=\int_U \nabla \dot{u}{\,\cdot\,} \nabla z_1\, dx=\int_{\Gam}\big[\div_{\Gamma} 
( (X{\,\cdot\,} \nu) \nabla_{\Gamma} u^-  )z^- -\div_{\Gamma} 
( (X {\,\cdot\,} \nu) \nabla_{\Gamma} u^+ )z^+\big]\, d\hn,
$$
where the last equality follows by integration by parts. This establishes the first part of the statement.
\medskip

\noindent
{\it Step 2. Computation of the first variation.} 
We shall show that 
\begin{equation} \label{1var}
\frac{d}{dt} F((u_{\Phi_{t}},K_{\Phi_{t}});U) 
= \int_{\Gamma_{\Phi_{t}}}(| \nabla_{\Gamma_{\Phi_{t}}} u^-_{\Phi_{t}} |^2 
- | \nabla_{\Gamma_{\Phi_{t}}} u^+_{\Phi_{t}} |^2 + H_{\Phi_{t}})( X_{\Phi_{t}} {\,\cdot\,} \nu_{\Phi_{t}} )
\, d\hn
\end{equation}
for every $t\in (-1, 1)$.

We start by performing a change of variables in the integral, which leads to
$$
\int_{U} |\nabla u_{\Phi_{t}}|^2\, dy = 
\int_{U}  |\nabla u_{\Phi_{t}} \circ \Phi_{t} |^2 
\det D\Phi_{t}\, dx  = \| \sqrt{ \det D\Phi_{t}  }\, 
( \nabla u_{\Phi_{t}} \circ \Phi_{t}) \|^2_{L^2 (U ;\R^N)}.
$$
By the regularity results of Proposition~\ref{prop:reg} and by the identity 
$$
\frac{\partial}{\partial t}(\det D\Phi_t)=(\div\, X_{\Phi_t}\circ \Phi_t)\det D\Phi_t
$$ 
(see \cite[Chapter III, Section~10]{gur} for a proof), we obtain 
\begin{eqnarray*}
\frac{d}{dt} \int_{U} |\nabla u_{\Phi_{t}}|^2\, dy 
& = &
\frac{d}{dt} \|\sqrt{ \det D\Phi_{t}  } \,
( \nabla u_{\Phi_{t}} \circ \Phi_{t}) \|^2_{L^2 (U;\R^N)}
\\
& = &
2\int_U \sqrt{ \det D\Phi_{t}  }
( \nabla u_{\Phi_{t}} \circ \Phi_{t}){\,\cdot\,}\frac{\partial}{\partial t}(
\sqrt{ \det D\Phi_{t}}
( \nabla u_{\Phi_{t}} \circ \Phi_{t}))\, dx
\\
& = &
\int_{U} |\nabla u_{\Phi_{t}} \circ \Phi_{t}|^2 
( \div\, X_{\Phi_{t}} \circ \Phi_{t} )\det D\Phi_{t} \, dx
\\
& & {}+ 2 \int_{U} (\nabla \dot{u}_{\Phi_{t}} \circ \Phi_{t}) 
{\,\cdot\,} (\nabla u_{\Phi_{t}} \circ \Phi_{t}) \det D \Phi_{t} \, dx 
\\
& & {}+ 2 \int_{U}
(\nabla^2 u_{\Phi_{t}} \circ \Phi_{t})[\nabla u_{\Phi_{t}} \circ \Phi_{t},
 \dot{\Phi}_{t} ] \det D\Phi_{t} \, dx 
\\
& = & \int_{U} |\nabla u_{\Phi_{t}} |^2 
 \div\, X_{\Phi_{t}} \, dy
+2 \int_{U} \nabla \dot{u}_{\Phi_{t}}  
{\,\cdot\,} \nabla u_{\Phi_{t}}  \, dy 
+ 2 \int_{U}
\nabla^2 u_{\Phi_{t}}[\nabla u_{\Phi_{t}} , X_{\Phi_{t}} ] \, dy 
\\
& = & \int_{U} 
\div ( |\nabla u_{\Phi_{t}}|^2 X_{\Phi_{t}} ) \, dy 
+ 2 \int_{U} \nabla \dot{u}_{\Phi_{t}}  
{\,\cdot\,} \nabla u_{\Phi_{t}}  \, dy 
= \int_{U} 
\div ( |\nabla u_{\Phi_{t}}|^2 X_{\Phi_{t}} ) \, dy, 
\end{eqnarray*}
where the last equality follows from (\ref{uphi}), since $\dot{u}_{\Phi_{t}}
\in L^{1,2}_0(U\setmeno K_{\Phi_t}; \partial U)$. 
Integrating by parts we deduce
$$
\frac{d}{dt}  \int_{U} |\nabla u_{\Phi_{t}}|^2\, dy = \int_{\G_{\Phi_{t}}} 
( | \nabla_{\G_{\Phi_{t}}} u^-_{\Phi_{t}} |^2 - | \nabla_{\G_{\Phi_{t}}} u^+_{\Phi_{t}} |^2 )
( X_{\Phi_{t}} {\,\cdot\,} \nu_{\Phi_{t}} ) \, d\hn, 
$$
which, together with \eqref{1varareat}, gives \eqref{1var}. 
\medskip

\noindent
{\it Step 3. Computation of the second variation.} 
We are now ready to compute \eqref{fvar2}.
To simplify the notation in the calculations below we set 
$f_{t}:= | \nabla u^-_{\Phi_{t}} |^2 - | \nabla u^+_{\Phi_{t}} |^2
+ H_{\Phi_{t}}$.
Using the fact that $|\nabla_{\Gamma_{\Phi_t}} u^\pm_{\Phi_{t}}|^2
=|\nabla u^\pm_{\Phi_{t}} |^2$ on $\Gamma_{\Phi_t}$, 
the area formula \eqref{farea}, and the identity 
$X_{\Phi_{t}} \circ \Phi_{t} = \dot{\Phi}_{t}$, we find
\begin{eqnarray}
\lefteqn{\frac{d^2}{d t^2} F ((u_{\Phi_{t}} , K_{\Phi_{t}});U)|_{t=0}
=  
\frac{d}{d t} \Big(
\int_{\Gamma_{\Phi_{t}}} f_{t} (X_{\Phi_{t}} {\,\cdot\,} \nu_{\Phi_{t}})
\,d\hn\Big)\Big|_{t=0}} 
\nonumber
\\
& = &
\frac{d}{d t} \Big(
\int_{\Gam}( f_{t} \circ \Phi_{t} ) (\dot{\Phi}_{t} 
{\,\cdot\,} ( \nu_{\Phi_{t}} \circ \Phi_{t} ) )J_{\Phi_{t}} \, d\hn
\Big)\Big|_{t=0}  
\nonumber
\\
& = &
\int_{\Gam} \frac{\partial}{\partial t}(f_{t} \circ \Phi_{t} )|_{t=0}
(X {\,\cdot\,} \nu)
\, d\hn
+ \int_{\Gam} f \frac{\partial}{\partial t}( 
\dot{\Phi}_{t} {\,\cdot\,} ( \nu_{\Phi_{t}} \circ \Phi_{t} ) \, 
J_{\Phi_{t}})|_{t=0}
\,d\hn
\nonumber
\\
& =: & \vphantom{\int} I_1+I_2. \label{int0}
\end{eqnarray}
The first integral $I_1$ can be written as
\begin{eqnarray}
I_1 & = &
\int_{\Gam} \dot{f} ( X {\,\cdot\,} \nu ) \, d\hn
+ \int_{\Gam} ( \nabla f {\,\cdot\,} X ) ( X {\,\cdot\,} \nu ) 
\, d\hn 
\nonumber
\\
& = &
\int_{\Gam} \dot{f} ( X {\,\cdot\,} \nu ) \, d\hn
+ \int_{\Gam} ( \nabla f {\,\cdot\,} \nu ) ( X {\,\cdot\,} \nu )^2 \, d\hn
\nonumber
\\
& &
{}
+ \int_{\Gam} ( \nabla_{\Gamma} f {\,\cdot\,} X^{\parallel} ) ( X {\,\cdot\,} 
\nu )
\, d\hn. \label{integr1}
\end{eqnarray}
By property (g) of Lemma~\ref{lem:identita}  
the second integral $I_2$ turns out to be
\begin{equation}\label{integr2} 
 I_2= \int_{\Gam} f 
( Z {\,\cdot\,} \nu 
-2 X^{\parallel} {\,\cdot\,} 
\nabla_{\Gamma} ( X {\,\cdot\,} \nu )
+ \B[X^{\parallel}, X^{\parallel}]) \,d\hn
+ \int_{\Gam} f \div_{\Gamma}( (X {\,\cdot\,} \nu) X )\,
d\hn. 
\end{equation}
We note that by \eqref{fdiv} we have
\begin{eqnarray*}
\lefteqn{\int_{\Gam} f \div_{\Gamma} ((X {\,\cdot\,} \nu) X)
\, d\hn
+ \int_{\Gam} 
( \nabla_{\Gamma} f {\,\cdot\,} X^{\parallel} ) ( X {\,\cdot\,} \nu ) 
\, d\hn} 
\\
& = & \int_{\Gam} 
\div_{\Gamma} ( f (X {\,\cdot\,} \nu) X) \, d\hn
=  \int_{\Gam} f H (X {\,\cdot\,} \nu)^2 \, d\hn.
\end{eqnarray*}
Combining the previous identity with \eqref{int0}--\eqref{integr2} we obtain
\begin{eqnarray}
\frac{d^2}{d t^2} F ((u_{\Phi_{t}} , K_{\Phi_{t}});U)|_{t=0} 
& = & \int_{\Gam} f 
( Z {\,\cdot\,} \nu 
-2X^{\parallel} {\,\cdot\,} 
\nabla_{\Gamma} ( X {\,\cdot\,} \nu )
+\B[ X^{\parallel}, X^{\parallel}] +H (X {\,\cdot\,} \nu)^2
)\, d\hn \nonumber
\\
& & {}+ \int_{\Gam} \dot{f} ( X {\,\cdot\,} \nu )\, d\hn
+ \int_{\Gam} ( \nabla f {\,\cdot\,} \nu )( X {\,\cdot\,} \nu )^2 \,
d\hn.\label{int1}
\end{eqnarray}
Using the definition of $f$ and properties (d) and (e) of
Lemma~\ref{lem:identita}, the last term in the previous expression
can be written as
\begin{eqnarray}
\lefteqn{\int_{\Gam} ( \nabla f {\,\cdot\,} \nu )( X {\,\cdot\,} \nu )^2 
\,d\hn} \nonumber
\\
& = & \int_{\Gam} ( 2\nabla^2 u^{-} [ \nu , \nabla_{\Gamma} u^{-} ] 
- 2\nabla^2 u^{+} [ \nu , \nabla_{\Gamma} u^{+} ]
+\partial_{\nu} H ) (X {\,\cdot\,} \nu)^2\, d\hn \nonumber\\
& = & \int_{\Gam} (
2 \B [\nabla_{\Gamma} u^{+} , \nabla_{\Gamma} u^{+}]
- 2 \B [\nabla_{\Gamma} u^{-} , \nabla_{\Gamma} u^{-}] - |\B|^2) 
(X {\,\cdot\,} \nu)^2\, d\hn.\label{last}
\end{eqnarray}
Differentiating $f$ with respect to $t$, we obtain
\begin{equation}\label{letus}
\int_{\Gam} \dot{f} ( X {\,\cdot\,} \nu )\, d\hn
= \int_{\Gam} ( 2 \nabla_{\Gamma} u^- {\,\cdot\,} \nabla_{\Gamma} \dot{u}^-
- 2 \nabla_{\Gamma} u^+ {\,\cdot\,} \nabla_{\Gamma} \dot{u}^+ 
+\dot{H} ) ( X {\,\cdot\,} \nu ) \, d\hn.
\end{equation}
Integrating by parts, according to \eqref{fparts}, and using
\eqref{upto-solves} (see also Remark~\ref{rm:upto}), we deduce
\begin{eqnarray}
2\int_{\Gam} (\nabla_{\Gamma} u^\pm {\,\cdot\,} \nabla_{\Gamma} \dot{u}^\pm)
( X {\,\cdot\,} \nu ) \,d\hn 
& = &
- 2 \int_{\Gam}  \dot{u}^\pm 
\div_{\Gamma} ( ( X {\,\cdot\,} \nu ) \nabla_{\Gamma} u^\pm )\,
d\hn
\nonumber
\\ 
& = & - 2 \int_{\Gam} \dot{u}^\pm
\partial_{\nu} \dot{u}^\pm\,  d\hn. \label{letus2}
\end{eqnarray}
Since $\partial_{\nu}\dot\nu{\,\cdot\,}\nu=
-\dot\nu{\,\cdot\,}\partial_{\nu}\nu=0$ by \eqref{NU}, we have 
$\div\,\dot\nu=\div_{\Gamma}\dot\nu$ and in turn, by
\eqref{defH}, $\dot H=\div_{\Gamma}\dot\nu$. 
Hence, integrating by parts and using (f) of Lemma~\ref{lem:identita},
we deduce
\begin{eqnarray}
\int_{\Gam} \dot{H} ( X {\,\cdot\,} \nu ) \, d\hn
& = & \int_{\Gam} \div_{\Gamma}\dot\nu\, ( X {\,\cdot\,} \nu ) \, d\hn
= - \int_{\Gam} \dot{\nu} {\,\cdot\,} \nabla_{\Gamma} 
(X {\,\cdot\,} \nu) 
\, d\hn 
\nonumber\\
& = &
\int_{\Gam} ( ( D_{\Gamma} X )^{T} [ \nu ]
+ D_\Gamma \nu [ X ] ) {\,\cdot\,} \nabla_{\Gamma} ( X {\,\cdot\,} \nu ) 
\, d\hn \nonumber \\
& = &
\int_{\Gam} | \nabla_{\Gamma} ( X {\,\cdot\,} \nu ) |^2 
\, d\hn. \label{letus3}
\end{eqnarray} 
Combining \eqref{int1}--\eqref{letus3}, we 
obtain \eqref{fvar2} and we  
conclude the proof of the theorem.
\end{proof}

\begin{remark} \label{remt}
Let us fix $s\in (-1,1)$. We observe that the family of diffeomorphisms
$$
\tilde{\Phi}_{h} := \Phi_{s+h} \circ \Phi^{-1}_s
$$
is an admissible flow for $\Gamma_{\Phi_s}=\Phi_s(\Gam)$ in $U$
(one can always reparameterize  the ``time'' variable $h$ 
away from $0$ so that $\tilde{\Phi}_{h}$ is defined for all $h\in (-1,1)$)
and that $(\dot{\tilde{\Phi}}_h)|_{h=0} = X_{\Phi_s} $
and $(\ddot{\tilde{\Phi}}_h)|_{h=0} = Z_{\Phi_s}$. 
Applying Theorem~\ref{th:var2}, we deduce that 
\begin{eqnarray*}
\lefteqn{\frac{d^2}{dt^2}F((u_{\Phi_t}, K_{\Phi_t});U)|_{t=s}
=\frac{d^2}{dh^2}F((u_{\Phi_{s+h}}, \tilde\Phi_h(K_{\Phi_s}));U)|_{h=0}}
\\
& = & 2 \int_{\Gamma_{\Phi_s}} ( \dot{u}^+_{\Phi_s}
\partial_{\nu_{\Phi_s}} \dot{u}^+_{\Phi_s} 
- \dot{u}^-_{\Phi_s} \partial_{\nu_{\Phi_s}} \dot{u}^-_{\Phi_s} ) \, d\hn
+ \int_{\G_{\Phi_s}} |\nabla_{\G_{\Phi_s}} 
( X_{\Phi_s}{\,\cdot\,} \nu_{\Phi_s} ) |^2\, d\hn
\\
& & 
\hskip-2mm
{}+\int_{\Gamma_{\Phi_s}} ( X_{\Phi_s} {\,\cdot\,} \nu_{\Phi_s})^2 
( 2 \B_{\Phi_s}[\nabla_{\Gamma_{\Phi_s}} u^+_{\Phi_s}, 
\nabla_{\Gamma_{\Phi_s}} u^+_{\Phi_s}] 
- 2 \B_{\Phi_s}[\nabla_{\Gamma_{\Phi_s}} u^-_{\Phi_s}, 
\nabla_{\Gamma_{\Phi_s}} u^-_{\Phi_s}]
- |\B_{\Phi_s}|^2 )\,  d\hn
\\
& & 
\hskip-2mm
{}+\int_{\Gamma_{\Phi_s}}  f_s
 ( Z_{\Phi_s}{\,\cdot\,} \nu_{\Phi_s} 
-2 X^{\parallel}_{\Phi_s}{\,\cdot\,} 
\nabla_{\Gamma_{\Phi_s}} ( X_{\Phi_s}{\,\cdot\,} \nu_{\Phi_s} )
+ \B_{\Phi_s}[X^{\parallel}_{\Phi_s}, X^{\parallel}_{\Phi_s}]
+H_{\Phi_s}( X_{\Phi_s}{\,\cdot\,} \nu_{\Phi_s})^2)\, d\hn,
\end{eqnarray*}
where $f_s:=|\nabla u^-_{\Phi_s}|^2-|\nabla u^+_{\Phi_s}|^2+
H_{\Phi_s}$.
Moreover, $\dot u_{\Phi_s}$ belongs to 
$L^{1,2}_0(U\setmeno K_{\Phi_s};\partial U)$ and satisfies
\begin{multline}\nonumber
\int_U \nabla \dot{u}_{\Phi_s}{\,\cdot\,}\nabla z\, dx
\\
{}+\int_{\Gamma_{\Phi_s}}
(\div_{\Gamma_{\Phi_s}} 
( (X_{\Phi_s} {\,\cdot\,} \nu_{\Phi_s}) 
\nabla_{\Gamma_{\Phi_s}} u^+_{\Phi_s})z^+
-\div_{\Gamma_{\Phi_s}} 
( (X_{\Phi_s} {\,\cdot\,} \nu_{\Phi_s}) \nabla_{\Gamma_{\Phi_s}} u^-_{\Phi_s}  
)z^-)\, d\hn=0
\end{multline}
for every $z\in L^{1,2}_0(U\setmeno K_{\Phi_s}; \partial U)$.
\end{remark}

As already explained in the introduction, 
in the context of this paper critical points are partially regular admissible pairs which satisfy an additional transmission condition along the discontinuity set. 

\begin{definition}\label{critical}
Let $\Om$, $(u,K)$, $U$, and $\Gamma$ be as in Definition~\ref{flux}. We say that $(u,K)$ is  a
{\em critical point in $U$ with respect to $\Gamma$ } if 
\begin{equation}\label{ecritical}
H=|\nabla_{\Gamma}u^+|^2-|\nabla_{\Gamma}u^-|^2\qquad\text{on $\Gamma\cap U$.}
\end{equation}
\end{definition}
 
If $(u,K)$ is a critical point, then the expression of the second 
variation of $F$ at $(u,K)$ simplifies, as the function $f$ vanishes. 
We have therefore the following corollary.

\begin{corollary}\label{cor:crit}
In addition to the hypotheses of Theorem~\ref{th:var2} assume that $(u,K)$
is a critical point in $U$ with respect to $\Gamma$. Then 
\begin{multline}\nonumber
\frac{d^2}{dt^2}F((u_{\Phi_t}, K_{\Phi_t});U)|_{t=0}
= 2 \int_{\Gam} ( \dot{u}^+
\partial_{\nu} \dot{u}^+ 
- \dot{u}^- \partial_{\nu} \dot{u}^- ) \, d\hn 
+ \int_{\Gam} | \nabla_{\G} 
( X{\,\cdot\,} \nu ) |^2 d\hn
\\
{}+\int_{\Gam} ( X {\,\cdot\,} \nu)^2 
(  2 \B[\nabla_{\Gamma} u^+, \nabla_{\Gamma} u^+] 
- 2 \B[\nabla_{\Gamma} u^-, \nabla_{\Gamma} u^-]
- | \B |^2 ) \, d\hn. 
\end{multline}
\end{corollary}

\begin{remark}\label{rm:crit}
We note that, if $(u,K)$ is a critical point in $U$ with respect to $\Gamma$, 
then the second variation of $(u, K)$ in $U$ with respect to the flow 
$(\Phi_t)$ depends only on the normal component of the variation
$X{\,\cdot\,}\nu$.  Moreover,
as $\dot u$ depends linearly on $X{\,\cdot\,\nu}$, the second variation
becomes in this case a quadratic form in the variable $X{\,\cdot\,}\nu$.
\end{remark}

The previous corollary suggests the following definition. 
Given $\Om$, $(u,K)$, $U$, and $\Gamma$ as in 
Definition~\ref{flux}, we can consider the function 
$\delta^2\!F ((u,\Gamma);U):H^1_0(\Gamma\cap U)\to\R$ given by
\begin{multline}\label{qvar2}
\delta^2\!F ((u,\Gamma);U)[\varphi]
:=  2\int_{\Gam} (\vf^+\partial_\nu\vf^+ - \vf^-\partial_\nu\vf^- )\, d\hn
+\int_{\Gam} |\nablat \varphi |^2\, d\hn 
\\
{}+\int_{\Gam} (2\B[\nablat u^+,\nablat u^+]
-2\B[\nablat u^-,\nablat u^-] -|\B|^2) \varphi^2\, d\hn, 
\end{multline}
where $\vf\in L^{1,2}_0(U\setmeno K; \partial U)$ is the solution of 
\begin{equation} \label{vphi}
\int_U \nabla \vf{\,\cdot\,} \nabla z\, dx+\int_{\Gam} (\div_{\Gamma} 
( \varphi \nabla_{\Gamma} u^+  )z^+-\div_{\Gamma} 
( \varphi\nabla_{\Gamma} u^-  )z^- )\, d\hn=0
\end{equation}
for all $z\in L^{1,2}_0(U\setmeno K; \partial U)$.
As $\vf$ depends linearly on $\varphi$, the function 
$\delta^2\!F ((u,\Gamma);U)$ defines a quadratic form on $H^1_0(\Gamma\cap U)$.
Arguing as in Remark~\ref{rm:upto}, it is easy to see that
\begin{equation}\label{mj1}
\partial_\nu \vf^\pm = \div_{\Gamma} 
(\varphi \nabla_{\Gamma} u^\pm) \qquad \text{on } \Gam
\end{equation}
and 
\begin{equation}\label{mj2}
\int_{\Gam} (\vf^+\partial_\nu \vf^+
-\vf^-\partial_\nu \vf^-)\, d\hn = - \int_U
|\nabla \vf|^2\, dx.
\end{equation}

We conclude this section by proving a second order necessary condition
for minimality, expressed in terms of the quadratic form defined in \eqref{qvar2}.
Minimality is intended in the sense of the following definition.

\begin{definition}\label{def:wmin}
Let $\Om$, $(u,K)$, $U$, and $\Gamma$ be as in 
Definition~\ref{flux} and let $k\in \N\cup\{\infty\}$.
We say that $(u,K)$ is a {\em $C^k$-local minimizer in $U$ with respect to $\Gamma$} if 
there exists $\delta>0$ such that 
\begin{equation}\label{locmin}
\int_U |\nabla u|^2\, dx +\hn(K\cap U) \leq
\int_U |\nabla v|^2\, dx +\hn(\Phi(K\cap U))
\end{equation}
for every $C^k$-diffeomorphism $\Phi$ on $\overline U$ with
$\Phi=I$ on $(K\cap U)\setmeno \Gamma$ and  $\|\Phi - I\|_{C^k}\le\delta$, and 
every $v\in L^{1,2}(U\setmeno \Phi(K\cap U))$ with $v=u$ $\hn$-a.e.\ on $\partial U$.
We say that $(u,K)$ is an {\em isolated $C^k$-local minimizer in $U$ with respect to $\Gamma$} if 
\eqref{locmin} holds with the strict inequality for
every $\Phi$ as before, with $\Gamma_\Phi\neq \Gam$.
\end{definition}

Every $C^{\infty}$-local minimizer has nonnegative second variation, as made
precise by the following proposition.

\begin{theorem}\label{th:necessary}
Assume that  $(u,K)$ is a $C^{\infty}$-local minimizer in $U$ with respect
to $\Gamma$. Then the quadratic form \eqref{qvar2} is positive semidefinite; 
i.e., 
\begin{equation}\label{name}
\delta^2\!F((u,\Gamma);U)[\varphi]\geq0 \qquad\text{for every }
\varphi\in H^1_0(\Gamma\cap U).
\end{equation}
\end{theorem}

\begin{proof}
Let us fix $\varphi\in C^{\infty}_c(\Gam)$ and
consider an admissible flow $(\Phi_t)$ for $\Gamma$ in $U$ 
such that for $t$ small enough $\Phi_t=(I+t\varphi\nu)\circ\Pi_{\Gamma}$ 
in a neighbourhood of $\Gam$, where $\Pi_{\Gamma}$ denotes the orthogonal 
projection on $\Gamma$.
It turns out that the vector field $X$, introduced in \eqref{noT}, 
coincides with $\varphi\nu$ on $\Gam$. Using Corollary~\ref{cor:crit} and 
the minimality of $(u,K)$ we  then deduce
$$
\delta^2\!F((u,\Gamma);U)[\varphi]= \frac{d^2}{dt^2}F((u_{\Phi_t}, K_{\Phi_t});U)|_{t=0}\geq 0.
$$
The thesis follows by approximating any $\varphi\in H^1_0(\Gam)$ with
functions in $C^{\infty}_c(\Gam)$.
\end{proof}

\end{section}

\begin{section}{Equivalent formulations of the second order condition}\label{sec:eq}

Throughout the whole section $(u,K)$ will be a pair in $\Ar(\Om)$ and $U\subset\Om$ an admissible
subdomain for $(u,K)$ in the sense of Definition~\ref{def:partreg}, while  
$\Gamma$ will denote a relatively open set compactly contained in $\Gamma_r$.

The purpose of this section is to perform a more detailed study of the second 
variation. In particular we shall derive some necessary and sufficient
conditions for the second variation to be  {\em positive definite} in admissible  subdomains $U$ of $\Om$; i.e.,
\begin{equation}\label{var2+}
\delta^2\!F((u,\Gamma); U)[\varphi]>0\quad\text{for every $\varphi\in H^1_0(\Gam)\setmeno\{0\}$.}
\end{equation}

In the first subsection we show that \eqref{var2+} is equivalent to a condition on the first eigenvalue of a suitable compact
operator $\tu$ on $H^1_0(\Gam)$.
In the second subsection we formulate  \eqref{var2+}
in terms of a dual minimum problem. 

\subsection{An equivalent eigenvalue problem}
We introduce a bilinear form on $H^1_0(\Gam)$ 
defined by
\begin{equation}\label{bilf}
(\varphi,\psi)_\sim:=\int_{\Gam} a\, \varphi\psi\, d\hn
+\int_{\Gam} \nablat\varphi {\,\cdot\,} \nablat\psi\, d\hn
\end{equation}
for every $\varphi,\psi\in H^1_0(\Gam)$, where
$$
a(x):=2 \B(x)[\nablat u^+(x),\nablat u^+(x)]
-2 \B(x)[\nablat u^-(x),\nablat u^-(x)] -|\B(x)|^2
$$
for every $x\in\Gam$. 

\begin{remark}\label{agen}
All the results contained in this section do not depend on the special form 
of $a$ and continue to hold whenever $a$ is replaced by any smooth bounded function on~$\Gamma$.
\end{remark}

We start by showing that the bilinear form \eqref{bilf}, when it is a scalar product,
is indeed equivalent to the standard scalar product of $H^1_0(\Gam)$.

\begin{proposition}\label{prop:ps}
Assume that
\begin{equation}\label{hyp1}
(\varphi,\varphi)_\sim> 0\qquad\text{for every $\varphi\in H^1_0(\Gam)\setmeno\{0\}$.}
\end{equation}
 Then the bilinear form \eqref{bilf}
defines an equivalent scalar product on $H^1_0(\Gam)$.
\end{proposition}

\begin{proof}
Assumption \eqref{hyp1} immediately implies that the bilinear form \eqref{bilf} is a scalar product.
In particular, 
\begin{equation}\label{tildenorm}
\|\varphi\|_\sim:=(\varphi,\varphi)_\sim^{1/2}
\end{equation}
defines a norm on $H^1_0(\Gam)$.

To show the equivalence with the scalar product of $H^1_0(\Gam)$,
we first observe that, as $a$ is bounded, we have
$\|\varphi\|_\sim\leq C\|\varphi\|_{H^1}$
for every $\varphi\in H^1_0(\Gam)$.
For the opposite inequality we argue by contradiction assuming that there exists a sequence
$(\varphi_n)$ such that $\|\varphi_n\|_{H^1}=1$ and
\begin{equation}\label{unitH1}
\|\varphi_n\|_\sim\leq \tfrac1n.
\end{equation} 
Then, up to subsequences, $\varphi_n\wto \varphi$ weakly in $H^1_0(\Gam)$. 
In particular, $\varphi_n\to \varphi$ in $L^2(\Gam)$, hence
\begin{equation}\label{conv}
\begin{array}{c}
\displaystyle
\int_\Gam a\varphi^2\,d\hn=\lim_n\int_\Gam a\varphi_n^2\,d\hn, 
\smallskip\\
\displaystyle
\int_\Gam |\nablat\varphi|^2\,d\hn\leq\liminf_n\int_\Gam |\nablat \varphi_n|^2\,d\hn. 
\end{array}
\end{equation} 
Recalling \eqref{unitH1} it follows that $\|\varphi\|_\sim=0$, that is, $\varphi=0$.
Using again \eqref{unitH1} and \eqref{conv}, we deduce that
$\int_\Gam |\nablat \varphi_n|^2\,d\hn\to 0$, which contradicts
the fact that $\|\varphi_n\|_{H^1}=1$.
\end{proof}

Given $\varphi\in H^1_0(\Gam)$ let $\vf$ be the function defined in \eqref{vphi}. The linear map 
$$
\psi\in H^1_0(\Gam)\mapsto -2 \int_\Gam (\vf^+\divt (\psi\nablat u^+)
-\vf^-\divt (\psi\nablat u^-))\,d\hn,
$$
is continuous on $H^1_0(\Gam)$. If condition \eqref{hyp1} is satisfied,
then by Proposition~\ref{prop:ps} and by the Riesz Theorem there exists a unique element
$\tu\varphi\in H^1_0(\Gam)$
such that
\begin{equation}\label{tildeT}
(\tu\varphi,\psi)_\sim = -2 \int_\Gam (\vf^+\divt (\psi\nablat u^+)
-\vf^-\divt (\psi\nablat u^-))\,d\hn
\end{equation}
for every $\psi\in H^1_0(\Gam)$. By this definition and \eqref{mj1}
it turns out that
\begin{equation}\label{mj3}
\delta^2\!F((u,\Gamma); U)[\varphi]
= \|\varphi\|^2_\sim -(\tu\varphi,\varphi)_\sim
\end{equation}
for every $\varphi\in H^1_0(\Gam)$, provided \eqref{hyp1} is satisfied.

We now study some properties of the operator $T$.

\begin{proposition}\label{Tcpt}
Assume condition \eqref{hyp1}.
Then the linear operator $\tu: (H^1_0(\Gam),\sim)\to (H^1_0(\Gam),\sim)$,
defined by \eqref{tildeT}, is monotone, compact, and self-adjoint.
\end{proposition}

\begin{proof}
By \eqref{mj1} and \eqref{mj2} we obtain
$$
(T\varphi,\varphi)_\sim=-2 \int_\Gam (\vf^+\divt (\varphi\nablat u^+)
-\vf^-\divt (\varphi\nablat u^-))\,d\hn=2\int_U|\nabla \vf |^2\, dx \geq0,
$$
that is, $T$ is monotone.

Let $\varphi_n\wto \varphi$ weakly in $(H^1_0(\Gam),\sim)$.
Then Proposition~\ref{prop:ps} implies  that
$\divt (\varphi_n\nablat u^\pm)$ converges
to $\divt (\varphi\nablat u^\pm)$ weakly in $L^2(\Gam)$.
From \eqref{vphi} it follows that $v_{\varphi_n}\wto\vf$
weakly in $L^{1,2}_0(U\setmeno K;\partial U)$. By the compactness of the trace operator
we have that $v_{\varphi_n}^\pm$
(up to additive constants on the connected components of $U\setmeno K$
whose boundary does not meet $\partial U$)
converges to $\vf^\pm$ strongly in $L^2(\Gam)$. This is enough to
deduce from \eqref{tildeT} that $\tu$ is weakly continuous, hence continuous.

Taking $\varphi=\varphi_n$ and $\psi=\tu\varphi_n$ in \eqref{tildeT}, we obtain
that $\|\tu\varphi_n\|_\sim\to \|\tu\varphi\|_\sim$, which concludes the proof
of the compactness of $\tu$.

Using the Green identity
\begin{eqnarray*}
\lefteqn{\int_\Gam (\vf^+ \divt(\psi \nablat u^+)
-\vf^- \divt (\psi \nablat u^-))\,d\hn}
\\
& = &
\int_\Gam (v_\psi^+ \divt (\varphi \nablat u^+)
- v_\psi^- \divt (\varphi \nablat u^-))\,d\hn,
\end{eqnarray*}
it is easy to check that $\tu$ is self-adjoint.
\end{proof}

Under the assumptions of Proposition~\ref{Tcpt} we can define
\begin{equation}\label{max}
\lambda_1:=\max_{\|\varphi\|_\sim=1} (\tu\varphi,\varphi)_\sim=\|T\|_\sim.
\end{equation}
It is well known that $\lambda_1$ coincides with the first eigenvalue of $\tu$.
The following proposition gives an equivalent characterization of $\lambda_1$.

\begin{proposition}\label{conjpt}
Assume condition \eqref{hyp1} and consider the 
following auxiliary system in the unknown 
$(v,\varphi)\in L^{1,2}_0(U\setmeno K;\partial U){\times}
H^1_0(\Gam)$:
\begin{equation}\label{auxil0}
\begin{array}{l}
\displaystyle
\lambda\int_U \nabla v{\,\cdot\,} \nabla z\, dx
+\int_{\Gam} [\div_{\Gamma} 
( \varphi \nabla_{\Gamma} u^+  )z^+-\div_{\Gamma} 
( \varphi \nabla_{\Gamma} u^-  )z^-]\, d\hn=0,
\medskip
\\
\displaystyle
\int_{\Gam} \nablat\varphi {\,\cdot\,} \nablat\psi\, d\hn
+\int_{\Gam} a\, \varphi\psi\, d\hn
\\
\displaystyle
\hphantom{\int_{\Gam} a\, \varphi\psi\, d\hn}
{}+2\int_{\Gam}(\divt(\psi\nablat u^+) v^+-\divt(\psi\nablat u^-) 
v^-)\, d\hn=0
\end{array}
\end{equation}
for all $z\in L^{1,2}_0(U\setmeno K;\partial U)$ and for all $\psi\in H^1_0(\Gam)$.
Then $\lambda_1$ coincides with the greatest $\lambda$ such that \eqref{auxil0} 
admits a nontrivial solution $(v,\varphi)\neq(0,0)$.
\end{proposition}

\begin{proof}
It is enough to observe that under condition \eqref{hyp1}, $\lambda$ is 
an eigenvalue of $\tu$ with eigenfunction $\varphi$ if and only if the pair $(\vf/\lambda,\varphi)$
(see \eqref{vphi} for the definition of $\vf$) is a nontrivial solution of \eqref{auxil0}.
\end{proof}

\begin{remark}
We note that the strong formulation of \eqref{auxil0} corresponds to
\begin{equation}\label{auxil}
\begin{cases}
\Delta v =0 & \text{in } U\setmeno K, \\
v=0 & \text{on }\partial U, \\
\partial_\nu v^\pm=0 &\text{on } K\setmeno (\Gam),\\
\lambda\,\partial_\nu v^\pm=\divt (\varphi\nablat u^\pm)
& \text{on }\Gamma\cap U, \\
-\Deltat \varphi+a\varphi=2\nablat u^+{\,\cdot\,}\nablat v^+
-2\nablat u^-{\,\cdot\,}\nablat v^- 
& \text{on }\Gamma\cap U.
\end{cases}
\end{equation}
\end{remark}

Condition \eqref{var2+} can be characterized in terms of $\lambda_1$, as explained in the following theorem.

\begin{theorem}\label{thm:eq1}
Condition \eqref{var2+} is satisfied if and only if
the following two properties hold:
\begin{itemize}
	\item[(i)] $(\varphi,\varphi)_\sim> 0$ for every $\varphi\in H^1_0(\Gam)\setmeno\{0\}$;
	\smallskip
	
	\item[(ii)] $\lambda_1<1$.
\end{itemize}
\end{theorem}

\begin{proof}
Assume that condition \eqref{var2+} is satisfied. Then by \eqref{mj2} we have
$$
(\varphi,\varphi)_\sim > -2
\int_\Gam (\vf^+ \partial_\nu \vf^+ 
-\vf^- \partial_\nu \vf^-)\,d\hn
=2\int_U |\nabla \vf|^2\, dx\geq0
$$
for every $\varphi\in H^1_0(\Gam)\setmeno\{0\}$, which implies condition (i).
Once (i) is satisfied, condition (ii) is equivalent to \eqref{var2+} by
\eqref{mj3} and \eqref{max}.
\end{proof}

Upon assuming \eqref{hyp1}, we can also characterize the positive
semidefiniteness of the second variation $\delta^2\!F((u,\Gamma);U)$
in terms of $\lambda_1$. More precisely,
we have the following.

\begin{theorem}\label{newthm}
Assume \eqref{hyp1}.
Then condition \eqref{name} holds if and only if $\lambda_1\leq 1$.
\end{theorem}

\begin{proof}
The fact that $\lambda_1\leq 1$ implies \eqref{name} follows from \eqref{mj3} and \eqref{max},
as in Theorem~\ref{thm:eq1}. Conversely, assume $\lambda_1>1$ and let $\varphi_1$ be an eigenfunction of $T$ associated with $\lambda_1$. Then by \eqref{mj3} we have
$\delta^2\!F((u,\Gamma);U)[\varphi_1]=(1-\lambda_1)\|\varphi_1\|_\sim^2<0$.
\end{proof}

We conclude this subsection with a corollary, where we show that 
pointwise coercivity of the second variation $\delta^2\!F((u,\Gamma);U)$
implies uniform coercivity.

\begin{corollary}\label{unif}
Assume \eqref{var2+}. Then there exists a constant $C>0$ such that
$$
\delta^2\!F((u,\Gamma); U)[\varphi]\geq C\|\varphi\|_{H^1}^2
$$
for every $\varphi\in H^1_0(\Gam)$.
\end{corollary}

\begin{proof}
We first note that by \eqref{mj3}
$$
\delta^2\!F((u,\Gamma); U)[\varphi]= \|\varphi\|^2_\sim -(\tu\varphi,\varphi)_\sim
\geq \|\varphi\|^2_\sim -\|\tu\|_\sim\|\varphi\|^2_\sim
=(1-\lambda_1)\|\varphi\|^2_\sim.
$$
The conclusion follows from Theorem~\ref{thm:eq1} and Proposition~\ref{prop:ps}.
\end{proof}

\begin{remark}\label{rmk:2d}
If $N=2$ condition \eqref{hyp1} is always true. Indeed, by \eqref{ecritical} 
the expression of $a(x)$ reduces to
$H^2(x)$. Therefore, by Theorem~\ref{thm:eq1} condition \eqref{var2+} is satisfied 
in this case if and only if $\lambda_1<1$.
In higher dimensions the situation is different. A counterexample can be constructed by 
considering as $\Gamma_r$ an unstable minimal hypersurface (i.e., a critical point of the 
area functional with nonpositive second variation) and then by choosing any function $u$ 
defined in a tubular neighbourhood of $\Gamma_r$, satisfying the first order conditions 
\eqref{uequiv} and \eqref{ecritical}, and 
$\nabla u^-=\nabla u^+$ on $\Gamma_r$.
This can be easily done using Cauchy-Kowalevskaya theorem. The conclusion follows by observing
that in this situation the bilinear form \eqref{bilf} reduces to the second variation of the area functional
at $\Gamma_r$.
\end{remark}

\subsection{A dual minimum problem}\label{subs:eig}
We introduce the linear operators
$$
A_\pm: H^1_0(\Gam)\to L^2(\Gam),
\qquad A_\pm\varphi:=-2\divt (\varphi\nablat u^\pm)
$$
and we denote by $A_\pm^*:L^2(\Gam)\to H^{-1}(\Gam)$ the adjoint
operators of $A_\pm$ with respect to the scalar product of $L^2(\Gam)$; i.e.,
for every $\psi\in L^2(\Gam)$ and every $\varphi\in H^1_0(\Gam)$
$$
\langle A_\pm^*\psi,\varphi\rangle=
\int_\Gam A_\pm\varphi\,\psi\, d\hn=
-2\int_\Gam \divt (\varphi\nablat u^\pm)\psi\, d\hn,
$$
where $\langle\cdot,\cdot\rangle$ denotes the duality product in 
$H^{-1}(\Gam){\times}H^1_0(\Gam)$.
We consider also the resolvent operator $R: H^{-1}(\Gam)\to H^1_0(\Gam)$,
which maps any $f\in H^{-1}(\Gam)$ into the solution $\varphi\in H^1_0(\Gam)$
of the problem
$$
\begin{cases}
-\Deltat \varphi+ a\varphi=f & \text{in }\Gam,
\\
\varphi\in H^1_0(\Gam).
\end{cases}
$$
The operator $R$ is well defined under the assumptions of Proposition~\ref{prop:ps}.
We note also that the operator $\tu$, introduced in \eqref{tildeT}, can be written as
\begin{equation}\label{tildeT2}
\tu\varphi= R(A_+^* \vf^+ - A_-^* \vf^-)
\end{equation}
for every $\varphi\in H^1_0(\Gam)$, where $\vf$ is defined in \eqref{vphi}.

We introduce now the following dual minimum problem:
\begin{equation}\label{dual}
\min\Big\{2 \int_U |\nabla v|^2\, dx : \ v\in L^{1,2}_0(U\setmeno K;\partial U),  
\ \|R(A_+^* v^+ - A_-^* v^-)\|_\sim =1\Big\}.
\end{equation}
An argument similar to the one used in the proof of Proposition~\ref{Tcpt} shows that 
\begin{equation}\label{opera}
v\mapsto R(A_+^* v^+ - A_-^* v^-)\quad\text{is compact from $L^{1,2}_0(U\setmeno K;\partial U)$ to $H^1_0(\Gam)$.}
\end{equation}
Exploiting this remark, it is not difficult to prove 
that the problem \eqref{dual} admits a solution by the direct method of the Calculus of Variations. 

The following theorem, which is the main result of this subsection, provides
a characterization of condition \eqref{var2+} in terms of the dual problem \eqref{dual}.

\begin{theorem}\label{thm:eq2}
Assume condition \eqref{hyp1}. Then $\lambda_1=1/\mu$,  where $\mu$ is the value of \eqref{dual}.
Moreover, condition \eqref{var2+} is satisfied if and only if \eqref{hyp1}
holds and $\mu>1$.
\end{theorem}

\begin{proof}
It is enough to prove that under \eqref{hyp1} we have $\mu=1/\lambda_1$, as 
the second part of the statement will then follow by Theorem~\ref{thm:eq1}.

Let $\varphi\in H^1_0(\Gam)$ be such that $\|\varphi\|_\sim=1$ and
$\tu\varphi=\lambda_1\varphi$. Then by \eqref{tildeT2} we have
\begin{equation}\label{11}
R(A_+^* \vf^+ - A_-^* \vf^-)=\lambda_1\varphi,
\end{equation}
that is
$$
-\lambda_1\Deltat\varphi+\lambda_1 a\varphi=A_+^* \vf^+ - A_-^* \vf^-.
$$
Multiplying both sides by $\varphi$ and integrating by parts, we obtain
$$
\lambda_1\int_\Gam a\varphi^2\, d\hn 
+\lambda_1\int_\Gam |\nablat\varphi|^2\, d\hn
= \int_\Gam(\vf^+A_+\varphi - \vf^-A_-\varphi)\,d\hn.
$$
Using the fact that $\|\varphi\|_\sim=1$ and $2\partial_\nu \vf^\pm=-A_\pm\varphi$,
we deduce that
\begin{equation}\label{13}
\lambda_1=-2\int_\Gam (\vf^+\partial_\nu \vf^+ - \vf^-\partial_\nu \vf^-)\, d\hn
=2\int_U |\nabla \vf|^2\, dx,
\end{equation}
where the last equality follows from \eqref{mj2}.
By \eqref{11} the function $\vf/\lambda_1$ is admissible for problem \eqref{dual}. Therefore,
from \eqref{13} we infer that $\mu\leq 1/\lambda_1$.

To show the converse inequality, let $v$ be a solution of \eqref{dual}. Then it
is easy to see that there exists a Lagrange multiplier $\mu_0$ such that
\begin{equation}\label{lag}
\int_U \nabla v{\,\cdot\,} \nabla z= \mu_0\, (R(A_+^*v^+-A_-^*v^-), R(A_+^*z^+-A_-^*z^-))_\sim
\end{equation}
for every $z\in L^{1,2}_0(U\setmeno K;\partial U)$.
Choosing $v$ as test function in \eqref{lag}, we deduce that $2\mu_0=\mu$.

We set $\varphi:=R(A_+^*v^+-A_-^*v^-)$ and $\psi:=R(A_+^*z^+-A_-^*z^-)$.
Then using the definition of $\psi$ and integrating by parts it turns out that
$$
\int_\Gam a\,\varphi\psi\,d\hn +\int_\Gam\nablat\varphi
{\,\cdot\,}\nablat\psi\,d\hn =\int_\Gam (z^+A_+\varphi -z^-A_-\varphi)\,d\hn,
$$
in other words
\begin{equation}\label{14}
(R(A_+^*v^+-A_-^*v^-), R(A_+^*z^+-A_-^*z^-))_\sim=
(\varphi,\psi)_\sim=
\int_\Gam (z^+A_+\varphi -z^-A_-\varphi)\,d\hn.
\end{equation}
From \eqref{lag} and \eqref{14} it follows that 
$\frac{1}{\mu}v$ satisfies \eqref{vphi},
which implies that $v=\mu\vf$.
Therefore, by \eqref{tildeT2} we have that 
$$
\tu\varphi=\frac{1}{\mu}R(A_+^*v^+-A_-^*v^-)=
\frac{1}{\mu}\varphi;
$$
i.e., $1/\mu$ is an eigenvalue of $\tu$. This implies that $1/\mu\leq\lambda_1$
and concludes the proof of the theorem.
\end{proof}

In the next corollary the dependence of $\lambda_1$ and $\mu$ on the domain will be made explicit.
In particular we will show that they depend monotonically on $U$.

\begin{corollary}\label{c:mono}
Assume that condition \eqref{hyp1} is satisfied. Let $U_1,U_2\subset\Omega$ be admissible subdomains
for $(u,K)$ such that $U_1\subset U_2$ and $\Gamma\cap U_1=\Gamma\cap U_2$.
Then $\lambda_1(U_1)\leq\lambda_1(U_2)$.
In particular, if condition \eqref{var2+} is satisfied in $U_2$, then it also holds in $U_1$.
\end{corollary}

\begin{proof}
As $\partial U_1\capset \partial_L(U_1\setmeno K)$,
 we have that if $v\in L^{1,2}_0(U_1\setmeno K;\partial U_1)$, then the function
$\tilde v$ given by $\tilde v=v$ on $U_1\setmeno K$ and $\tilde v=0$ on $U_2\setmeno U_1$
belongs to $L^{1,2}_0(U_2\setmeno K;\partial U_2)$. Therefore, 
if $v$ is an admissible function for the problem \eqref{dual} in $U_1$,
then $\tilde v$ is an admissible function for the problem \eqref{dual} in $U_2$.
Hence $\mu(U_1)\geq\mu(U_2)$.
The conclusions follows from Theorems~\ref{thm:eq2} and~\ref{thm:eq1}.
\end{proof}
The following corollary will be used in the next section (see Remark~\ref{r:mono2}). It shows that 
$\lambda_1$ and $\mu$ are continuous along decreasing sequences of open sets.
 
\begin{corollary}\label{c:mono2}
Assume that condition \eqref{hyp1} is satisfied. 
Let $U_n\subset\Om$ be a decreasing sequence of admissible subdomains for $(u,K)$.
Assume also that the open set $U$ defined as the interior part of $\cap_{n=1}^\infty U_n$
is an admissible subdomain for $(u, K)$ and that 
$\Gamma\cap U_n=\Gamma\cap U$ for every $n$. Then $\lambda_1(U_n)\to \lambda_1(U)$.
\end{corollary} 

\begin{proof}
In view of Corollary \ref{c:mono} it is enough to show that $\lim_{n}\lambda_1(U_n)\leq \lambda_1(U)$.
By Theorem~\ref{thm:eq2} this is equivalent to prove that
\begin{equation}\label{byeq2}
\lim_{n}\mu(U_n)\geq\mu(U).
\end{equation}
Let $v_n$ be a solution of \eqref{dual} with $U$ replaced by $U_n$. Then the function $\tilde v_n$ 
given by $\tilde v_n=v_n$ on $U_n$ and $\tilde v_n=0$ on $U_1\setmeno U_n$ belongs
to $L^{1,2}_0(U_1\setmeno K;\partial U_1)$ and
$2\int_{U_1}|\nabla \tilde v_n|^2\, dx=\mu(U_n)\leq \mu(U)$. 
Hence there exists a subsequence (not relabelled) $\tilde v_n$ and a function 
$\tilde v\in L^{1,2}_0(U_1\setmeno K;\partial U_1)$ such that $\tilde v_n\wto\tilde v$ weakly in
$L^{1,2}_0(U_1\setmeno K;\partial U_1)$. Clearly $\tilde v =0$ a.e.\ in $U_1\setmeno U$, which in turn implies 
that the restriction $v$ of $\tilde v$ to the set $U$ belongs to $L^{1,2}_0(U\setmeno K;\partial U)$.
Recalling also \eqref{opera} we infer that $v$ is admissible for problem \eqref{dual} and thus
$$
\lim_{n}\mu(U_n)=\lim_n\, 2\int_{U_1}|\nabla \tilde v_n|^2\, dx\geq 2\int_{U_1}|\nabla \tilde v|^2\, dx
   =2\int_{U}|\nabla v|^2\, dx\geq \mu(U),
$$
which shows \eqref{byeq2} and concludes the proof.
\end{proof}
\end{section}

\begin{section}{A second order sufficient minimality condition}\label{sec:suf}

In this section we show that any critical point satisfying the second order
condition \eqref{var2+} is a local minimizer with respect to variations of class $C^2$
of the regular part $\Gamma_r$ of the discontinuity set.
Critical points which are $C^2$-local minimizers (in the sense
of Definition~\ref{def:wmin}) play in our context the same role of weak
minimizers in the classical Calculus of Variations, as made precise by the following theorem.

\begin{theorem}\label{thm:suff}
Let $N\leq3$ and let $\Om$, $(u,K)$, $U$, and $\Gamma$ be as in Definition~\ref{flux}.
Assume in addition that $(u,K)$ is a critical point in $U$ with respect to $\Gamma$ and that  \eqref{var2+}
is satisfied. Then $(u,K)$ is an isolated $C^2$-local minimizer in $U$ with respect to $\Gamma$.
\end{theorem}

\begin{remark}\label{r:mono2}
We observe that in the statement of the theorem we can assume without loss of generality that 
$\Gamma\subset\subset U$. Indeed, if this is not the case, setting $\Gamma':=\Gam$, we can 
find an admissible subdomain $U'\supset U$ such that $\Gamma'\subset\subset U'$ and 
$\delta^2\!F((u,\Gamma'); U')$ is positive definite on $H^1_0(\Gamma')$. The existence of such a domain $U'$ 
is guaranteed by Corollary~\ref{c:mono2}. It is now sufficient to show that $(u,K)$ is an isolated $C^2$-local minimizer 
in $U'$ with respect to $\Gamma'$, since this implies in particular the thesis of Theorem~\ref{thm:suff}.
\end{remark}

In view of the previous remark we may assume in the remaining part of the section that 
$$
\Gamma\subset\subset U.
$$
In order to prove Theorem~\ref{thm:suff} we need some auxiliary results, which are contained in
the next lemmas. 
For every $\delta>0$ we define the $\delta$-neighbourhood $(A)_{\delta}$ of an arbitrary 
set $A\subset \R^N$ as
\begin{equation}\label{ingra}
(A)_{\delta}:=\{x\in \R^N:\, \dist(x,A)<\delta\}.
\end{equation}
For notational convenience we set
$$
{\mathcal D}_\delta:=\{\Phi \in C^2(\overline U;\overline U):\ \Phi \text{ diffeomorphism}, \
\Phi=I \text{ on } (K\cap U)\setmeno \Gamma,\ 0<\|\Phi - I\|_{C^2}\le\delta\}
$$ 
for every $\delta>0$. 
We fix $\delta_0>0$ such that the orthogonal projection $\Pi_{\Gamma\!_r}$ on $\Gamma_r$
is well defined (and smooth) in $(\Gamma)_{\delta_0}\cap U$ and  for every $\Phi\in {\mathcal D}_{\delta_0}$
there exists a unique $\varphi\in C^2_0(\Gamma)$ such that 
$$
\Gamma_\Phi=\Phi(\Gamma)=\{x+\varphi(x)\nu(x):\ x\in\Gamma \}.
$$ 
We can then define in  $(\Gamma )_{\delta_0}\cap U$ the vector field
\begin{equation}\label{xfi}
\tilde X_\Phi:=(\varphi\nu)\circ\Pi_{\Gamma\!_r}
\end{equation}
for every $\Phi\in {\mathcal D}_{\delta_0}$. 
Moreover, we consider the bilinear form
\begin{equation}\label{tnorm}
(\vartheta,\psi)_{\sim,\Phi}:=\int_{\Gamma_\Phi} a_\Phi \vartheta\psi\, d\hn +\int_{\Gamma_\Phi}
\nabla_{\Gamma_\Phi}\vartheta{\;\cdot\,}\nabla_{\Gamma_\Phi}\psi \, d\hn
\end{equation}
for every $\vartheta,\psi\in H^1_0(\Gamma_\Phi)$, where
$$
a_\Phi:=  2\B_\Phi[\nabla_{\Gamma_\Phi}u^+_\Phi, \nabla_{\Gamma_\Phi}u^+_\Phi]
-2\B_\Phi[\nabla_{\Gamma_\Phi}u^-_\Phi, \nabla_{\Gamma_\Phi}u^-_\Phi]
-|\B_\Phi|^2
$$
(here and in the sequel we use the same notation as in the previous sections).

In the next lemma we prove that the $H^1$-norm on $\Gamma_\Phi$ can be controlled in terms
of the norm $\|\cdot\|_{\sim,\Phi}$, uniformly with respect to $\Phi$.

\begin{lemma}\label{lm:disc}
There exist $C_1>0$ and  $\delta_1\in(0,\delta_0)$ such that 
for every $\Phi\in {\mathcal D}_{\delta_1}$ we have
\begin{equation}\label{eq48}
\|\psi\|_{H^1(\Gamma_\Phi)} \le C_1\|\psi\|_{\sim,\Phi}
\end{equation}
for every $\psi\in H^1_0(\Gamma_\Phi)$.
\end{lemma}

\begin{proof}
As $(u,K)$ satisfies the second order condition \eqref{var2+}, by Theorem~\ref{thm:eq1} and 
Proposition~\ref{prop:ps} we have that there exists a constant $C>0$ such that
\begin{equation}\label{eq48bis}
\|\psi\|_{H^1(\Gamma)}^2 \le C\|\psi\|_\sim^2
\end{equation}
for every $\psi\in H^1_0(\Gamma)$.
Setting $M:=\sup_{\Phi\in{\mathcal D}_{\delta_0}} \sup_{x\in\Gam} 
J_\Phi(x)\, (|(D_\Gamma \Phi)^{-T}(x)|^2+1)$,
by the area formula \eqref{farea} we obtain
\begin{eqnarray}
\|\psi\|_{H^1(\Gamma_\Phi)}^2 &=&
\int_\Gamma \big(|\psi\circ\Phi|^2+ |(\nabla_{\Gamma_\Phi}\psi) \circ\Phi|^2\big)J_\Phi\, d\hn
\nonumber
\\
& = & \int_\Gamma\big(|\psi\circ\Phi|^2+ |(D_\Gamma \Phi)^{-T} [\nabla_\Gamma (\psi\circ\Phi)]|^2 
\big)J_\Phi\, d\hn
\nonumber
\\
& \le & M \int_\Gamma \big(|\psi\circ\Phi|^2+ |\nabla_\Gamma (\psi\circ\Phi)|^2\big)\, d\hn
\le MC \|\psi\circ\Phi \|_\sim^2,
\label{farr}
\end{eqnarray}
where in the last inequality we used \eqref{eq48bis}.

Let $\e$ be a positive constant that will be chosen later.
By classical elliptic estimates (see, e.g., \cite[Theorem~3.17]{Tro}) we have that $u_\Phi^\pm$ is
$C^{1,\alpha}$ up to $\Gamma_\Phi$ for some $\alpha\in(0,1)$, with $C^{1,\alpha}$-norm uniformly bounded
with respect to $\Phi\in {\mathcal D}_{\delta_0}$. It follows that the map 
$\Phi\mapsto (a_\Phi\circ\Phi) \,J_\Phi$ 
is continuous from ${\mathcal D}_{\delta_0}$, endowed with the $C^2$ topology,
into $L^\infty(\Gamma)$. In particular, there esists $\delta_1\in(0,\delta_0)$ such that
$\|(a_\Phi\circ\Phi) \,J_\Phi-a \|_{L^\infty(\Gamma)} <\e$
for every $\Phi\in {\mathcal D}_{\delta_1}$, 
and, taking $\delta_1$ smaller, if needed, we can also guarantee that
$\sup_{\Phi\in{\mathcal D}_{\delta_1}} \sup_{x\in\Gam} J_\Phi^{-1}(x)\, |(D_\Gamma \Phi)^T(x)|^2
<1 +\e$ and $\sup_{\Phi\in{\mathcal D}_{\delta_1}} \sup_{x\in\Gam} J_\Phi^{-1}(x)<1+\e$.
Hence, using also the area formula \eqref{farea}, we have
\begin{eqnarray*}
\|\psi\circ\Phi \|_\sim^2
& \le &
\int_\Gamma \!\!(a_\Phi\circ\Phi)|\psi\circ\Phi|^2J_\Phi\, d\hn {+}\int_\Gamma\!\! |\nabla_\Gamma(\psi\circ\Phi)|^2\, d\hn
{+}\e \int_\Gamma\!\! |\psi\circ\Phi|^2\, d\hn 
\\
& \le &
\int_{\Gamma_\Phi} a_\Phi |\psi|^2 \, d\hn +(1+\e)\int_{\Gamma_\Phi} |\nabla_{\Gamma_\Phi}\psi
|^2\, d\hn +\e(1+\e)\int_{\Gamma_\Phi}  |\psi|^2 \, d\hn
\\
& \le &
\|\psi \|_{\sim,\Phi}^2 + \e(1+\e)\int_{\Gamma_\Phi}( |\psi|^2+ |\nabla_{\Gamma_\Phi}\psi
|^2)\, d\hn.
\end{eqnarray*}
Choosing $\e>0$ such that $MC\e(1+\e)=\frac12$, the thesis follows from
\eqref{farr} and the previous inequality with $C_1:=\sqrt{2MC}$.
\end{proof}

From the previous lemma, Proposition~\ref{prop:ps}, and Remark~\ref{agen} it follows that
for every $\Phi\in{\mathcal D}_{\delta_1}$
the bilinear form $(\cdot,\cdot)_{\sim,\Phi}$ is a scalar product on $H^1_0(\Gamma_\Phi)$,
so that, similarly to \eqref{tildeT}, we can introduce
the operator $T_\Phi: H^1_0(\Gamma_\Phi)\to H^1_0(\Gamma_\Phi)$   defined by
\begin{equation}\label{tildeTt}
(\tu_\Phi \vartheta,\psi)_{\sim,\Phi} = -2 \int_{\Gamma_\Phi} (\vP^+\div_{\Gamma_\Phi} (\psi\nabla_{\Gamma_\Phi} u_\Phi^+)
-\vP^-\div_{\Gamma_\Phi}(\psi\nabla_{\Gamma_\Phi} u_\Phi^-))\,d\hn,
\end{equation}
where $\vP \in L^{1,2}_0(U\setmeno K_\Phi;\partial U)$ is the solution of
$$
\int_U \nabla \vP{\,\cdot\,} \nabla z\, dx+\int_{\Gamma_{\Phi}} (\div_{\Gamma_{\Phi}} 
( \vartheta \nabla_{\Gamma_{\Phi}} u^+_{\Phi}  )z^+-\div_{\Gamma_{\Phi}} 
( \vartheta\nabla_{\Gamma_{\Phi}} u^-_{\Phi}  )z^- )\, d\hn=0
$$
for all $z\in L^{1,2}_0(U\setmeno K_{\Phi}; \partial U)$.
By Proposition~\ref{Tcpt} and Remark~\ref{agen} the operator
$T_\Phi$ is monotone, compact, and self-adjoint for every $\Phi\in{\mathcal D}_{\delta_1}$.
Moreover, we have the following property.

\begin{lemma}\label{lm:Pnorm}
Assume $N\leq3$.
For $\Phi\in{\mathcal D}_{\delta_1}$ let $\lambda_{1,\Phi}$ denote the norm of $T_\Phi$ on $H^1_0(\Gamma_\Phi)$ endowed with
the norm $\|\cdot\|_{\sim,\Phi}$ and let $\lambda_1:=\lambda_{1,I}$. Then
\begin{equation}\label{limsup}
\limsup_{\|\Phi- I\|_{C^2}\to 0}\,\lambda_{1,\Phi}\le \lambda_1.
\end{equation}
\end{lemma}

\begin{remark}
It is actually possible to prove that $\lambda_{1,\Phi}$ converges to $\lambda_1$, as 
$\|\Phi- I\|_{C^2}\to 0$, but this is not needed in the sequel.
\end{remark}

\begin{proof}[Proof of Lemma~\ref{lm:Pnorm}]
Assume by contradiction that \eqref{limsup} fails. 
Then there exist $\lambda_{\infty}>\lambda_1$, $\Phi_n\to I$ in $C^2$-norm,
$\varphi_n\in C^{\infty}_c(\Gamma_{\Phi_n})$ with $\|\varphi_n\|_{\sim,\Phi_n}=1$, 
and $w_n\in L^{1,2}_0(U\setmeno K_{\Phi_n};\partial U)$ solution to
$$
\int_U \nabla w_n{\,\cdot\,} \nabla z\, dx+\int_{\Gamma_{\Phi_n}} (\div_{\Gamma_{\Phi_n}} 
( \varphi_n \nabla_{\Gamma_{\Phi_n}} u^+_{\Phi_n}  )z^+-\div_{\Gamma_{\Phi_n}} 
( \varphi_n \nabla_{\Gamma_{\Phi_n}} u^-_{\Phi_n}  )z^- )\, d\hn=0
$$
for all $z\in L^{1,2}_0(U\setmeno K_{\Phi_n}; \partial U)$, 
such that 
\begin{equation}\label{absurd}
(T_{\Phi_n}\varphi_n, \varphi_n)_{\sim,\Phi_n}=2\int_U|\nabla w_n|^2\, dx\to \lambda_{\infty}>\lambda_1.
\end{equation}
Let $\tilde w_n:=w_n\circ\Phi_n$. Then 
$\tilde w_n\in L^{1,2}_0(U\setmeno K;\partial U)$ satisfies
\begin{multline}\nonumber
\int_U A_n[\nabla \tilde w_n,\nabla z]\, dx+\int_{\Gamma} (\div_{\Gamma_{\Phi_n}} 
( \varphi_n \nabla_{\Gamma_{\Phi_n}} u^+_{\Phi_n}  ))\circ\Phi_nJ_{\Phi_n}z^+\, d\hn
\\
{}-\int_{\Gamma}(\div_{\Gamma_{\Phi_n}} 
( \varphi_n \nabla_{\Gamma_{\Phi_n}} u^-_{\Phi_n}  ))\circ\Phi_nJ_{\Phi_n}z^-\, d\hn=0
\end{multline}
for all $z\in L^{1,2}_0(U\setmeno K; \partial U)$,
where $A_n:=\frac{D\Psi_nD\Psi_n^T}{{\det D\Psi_n}}\circ \Phi_n$ with $\Psi_n:=\Phi_n^{-1}$, 
while $J_{\Phi_n}$ is the $(N-1)$-dimensional Jacobian of $\Phi_n$. Moreover, it is easily seen that
\begin{equation}\label{easy}
\lim_n \,2\int_U |\nabla \tilde w_n|^2\, dx=\lim_n\, 2\int_U |\nabla w_n|^2\, dx =\lambda_{\infty}. 
\end{equation}
We finally set $\tilde \varphi_n:=c_n \varphi_n\circ \Phi_n$, where 
\begin{equation}\label{cn}
c_n:=\|\varphi_n\circ\Phi_n\|_{\sim}^{-1}\to 1,
\end{equation}
and we consider the function $v_{\tilde\varphi_n}$ defined by \eqref{vphi}
with $\varphi$ replaced by $\tilde\varphi_n$.
To conclude the proof of the lemma it will be enough to show that
\begin{equation}\label{enough}
\lim_n\int_U|\nabla(v_{\tilde\varphi_n}-\tilde w_n)|^2\, dx=0.
\end{equation}
Indeed, by \eqref{absurd} and \eqref{easy} this would imply
$$
\lambda_1\geq \lim_n\,(T\tilde \varphi_n, \tilde \varphi_n)_{\sim}
=\lim_n\, 2\int_U |\nabla v_{\tilde\varphi_n}|^2\, dx
=\lim_n\, 2\int_U |\nabla \tilde w_n|^2\, dx=\lambda_{\infty}>\lambda_1,
$$
which gives a contradiction.

In order to prove \eqref{enough} we observe that $z_n:=v_{\tilde\varphi_n}-\tilde w_n$
solves the problem
$$
\begin{array}{l}
z_n\in L^{1,2}_0(U\setmeno K; \partial U),
\vphantom{\displaystyle\int}\\
\displaystyle\int_UA_n [\nabla z_n,\nabla z]\, dx-\int_U (A_n-I)[\nabla v_{\tilde\varphi_n},\nabla z]\, dx
+\int_{\Gamma}\big(h_n^+z^+ -h_n^- z^-\big)\, d\hn=0
\end{array}
$$
for all $z\in L^{1,2}_0(U\setmeno K; \partial U)$,
where $h^\pm_n:=\div_{\Gamma}(\tilde \varphi_n\nabla_{\Gamma}u^\pm)
-\big( ( \div_{\Gamma_{\Phi_n}}(\varphi_n \nabla_{\Gamma_{\Phi_n}}u^\pm_{\Phi_n}) )
\circ\Phi_n\big)J_{\Phi_n}$. 
Since $A_n-I\to 0$ in $C^1$-norm and $v_{\tilde\varphi_n}$ is bounded in 
$L^{1,2}_0(U\setmeno K; \partial U)$, we have that $(A_n-I)[\nabla v_{\tilde\varphi_n}]$ 
converges to $0$ strongly in $L^2(U;\R^N)$.
Hence \eqref{enough} follows once we show that $h^\pm_n\to 0$ in $H^{-\frac12}(\Gamma)$.

To this aim let $\zeta\in C^{\infty}_c(\Gamma)$. Then we have
\begin{eqnarray*}
\lefteqn{\int_{\Gamma}\big( (\div_{\Gamma_{\Phi_n}}(\varphi_n \nabla_{\Gamma_{\Phi_n}}u^\pm_{\Phi_n}))
\circ\Phi_n\big)J_{\Phi_n}\zeta\, d\hn
=\int_{\Gamma_{\Phi_n}}\div_{\Gamma_{\Phi_n}}(\varphi_n\nabla_{\Gamma_{\Phi_n}}u^\pm_{\Phi_n})
(\zeta\circ\Psi_n)\, d\hn}
\\
& = &
{}-\int_{\Gamma_{\Phi_n}}\varphi_n\nabla_{\Gamma_{\Phi_n}}u^\pm_{\Phi_n}{\,\cdot\,}
\nabla_{\Gamma_{\Phi_n}}(\zeta\circ\Psi_n)\,d\hn
\\
& = &
{}-\int_{\Gamma_{\Phi_n}}\varphi_n(D_{\Gamma}\Phi_n)^{-T}\!\!\circ\Psi_n\,
[\nabla_{\Gamma}(u^\pm_{\Phi_n}\circ\Phi_n)\circ\Psi_n] {\,\cdot\,} 
(D_{\Gamma_{\Phi_n}}\Psi_n)^T[(\nabla_{\Gamma}\zeta)\circ\Psi_n]\,d\hn
\\
& = &
{}-\int_{\Gamma} c_n^{-1}\tilde\varphi_n
(D_{\Gamma}\Phi_n)^{-1}(D_{\Gamma}\Phi_n)^{-T}[\nabla_{\Gamma}(u^\pm_{\Phi_n}\circ\Phi_n),
\nabla_{\Gamma}\zeta] J_{\Phi_n}\, d\hn
\\
& = &
\int_{\Gamma} c_n^{-1}\div_{\Gamma}(\tilde\varphi_nJ_{\Phi_n}
(D_{\Gamma}\Phi_n)^{-1}(D_{\Gamma}\Phi_n)^{-T}[\nabla_{\Gamma}(u^\pm_{\Phi_n}\circ\Phi_n)]) 
\zeta\, d\hn,
\end{eqnarray*}
where we repeatedly used the area formula \eqref{farea}.
It follows that
\begin{equation}\label{hnbis}
h_n^\pm=\div_{\Gamma}(\tilde\varphi_n\nabla_{\Gamma}u^\pm- 
c_n^{-1}\tilde\varphi_nJ_{\Phi_n}
(D_{\Gamma}\Phi_n)^{-1}(D_{\Gamma}\Phi_n)^{-T}[\nabla_{\Gamma}(u^\pm_{\Phi_n}\circ\Phi_n)]).
\end{equation}

We claim that for every $\alpha\in(0,1)$
\begin{equation}\label{claimbis}
\nabla_{\Gamma}(u^\pm_{\Phi_n}\circ\Phi_n)\to \nabla_{\Gamma}u^\pm\quad
\text{in }C^{0,\alpha}(\overline\Gamma;\R^N).
\end{equation} 
To prove this we observe that $y_n:=u_{\Phi_n}\circ\Phi_n-u$ solves
$$
\begin{array}{l}
y_n\in L^{1,2}_0(U\setmeno K; \partial U),\vphantom{\displaystyle\int}
\\
\displaystyle\int_UA_n [\nabla y_n,\nabla z]\, dx
+\int_U (A_n-I)[\nabla u,\nabla z]\, dx=0
\end{array}
$$
for all $z\in L^{1,2}_0(U\setmeno K; \partial U)$. 
As $A_n\to I$ in $C^1(\overline U;\R^{N{\times}N})$, 
we deduce by
standard elliptic estimates (see, e.g., \cite[Theorem~3.17]{Tro}) 
that $y_n\to 0$ in $W^{2,p}(V\setmeno \Gamma)$ for every $p$ and for a suitable
neighbourhood $V$ of $\Gamma$. This provides \eqref{claimbis}.  

It is now convenient to set $\psi_n^\pm:=c_n^{-1}J_{\Phi_n}
(D_{\Gamma}\Phi_n)^{-1}(D_{\Gamma}\Phi_n)^{-T}[\nabla_{\Gamma}(u^\pm_{\Phi_n}\circ\Phi_n)]-\nabla_{\Gamma}u^\pm$. 
As the matrix $J_{\Phi_n}(D_{\Gamma}\Phi_n)^{-1}(D_{\Gamma}\Phi_n)^{-T}$ converge to $I$ in 
$C^1(\overline{\Gamma};\R^{N{\times}N})$, claim \eqref{claimbis} and 
the convergence in \eqref{cn} imply that for every $\alpha\in(0,1)$
\begin{equation}\label{psin0}
\psi_n^\pm\to 0 \quad
\text{in }C^{0,\alpha}(\overline\Gamma;\R^N).
\end{equation}
Let us fix $\alpha\in(0,1)$ and $p>1$ such that $(2\alpha-1)p>2N-2$. 
As $\tilde \varphi_n$ is bounded in $H^1_0(\Gamma)$ and $N\leq3$, by the
Sobolev imbedding theorem $\tilde \varphi_n$ is bounded in $L^p(\Gamma)$, too.
Adding and subtracting the term $\tilde\varphi_n(x)\psi_n^\pm(y)$ and
using the H\"older continuity of $\psi^\pm_n$, we can estimate the Gagliardo
$H^{\frac12}$-seminorm of $\tilde\varphi_n\psi_n^\pm$ as follows:
\begin{eqnarray*}
\lefteqn{\int_\Gamma\!\int_\Gamma \frac{|\tilde\varphi_n(x)\psi_n^\pm(x)-
\tilde\varphi_n(y)\psi_n^\pm(y)|^2
}{|x-y|^N}\, d\hn\!(x)\,d\hn\!(y)} 
\\
& \leq &
2\|\psi_n^\pm\|_{L^\infty(\Gamma)}^2\|\tilde\varphi_n\|_{H^{1\!/2}(\Gamma)}^2
+ 2 \|\psi_n^\pm\|_{C^{0,\alpha}(\Gamma)}^2 
\int_\Gamma\! \int_\Gamma |\tilde\varphi_n(x)|^2 |x-y|^{2\alpha-N}\, d\hn\!(x)\,d\hn\!(y)
\\
& \leq &
2\|\psi_n^\pm\|_{C^{0,\alpha}(\Gamma)}^2\Big(
\|\tilde\varphi_n\|_{H^{1\!/2}(\Gamma)}^2
\\
& &
\hphantom{2\|\psi_n^\pm\|_{C^{0,\alpha}(\Gamma)}^2\Big(} + \hn(\Gamma)^{\frac{2}{p}}\|\tilde\varphi_n\|_{L^p(\Gamma)}^2
\Big(\int_\Gamma\!\int_\Gamma |x-y|^{(2\alpha-N)\frac{p}{p-2}}\, d\hn\!(x)\,d\hn\!(y)\Big)^{\!\frac{p-2}{p}}\Big).
\end{eqnarray*}
By our choice of $\alpha$ and $p$ the last integral in the previous formula is finite. Thus, using
the boundedness of $\tilde\varphi_n$ in $H^{\frac12}(\Gamma)$ and in $L^p(\Gamma)$,
we deduce from \eqref{psin0} that
$$
\tilde \varphi_n\psi_n^\pm\to 0
\quad\text{in } H^{\frac12}(\Gamma;\R^N),
$$
which in turn gives $h^\pm_n\to 0$ in $H^{-\frac12}(\Gamma)$ by 
the definition of $\psi_n^\pm$. 
\end{proof}

\begin{remark}\label{rmk:N=4}
The assumption $N\leq3$ in Lemma~\ref{lm:Pnorm} can be removed if we require
$\Phi$ to converge to $I$ in the $C^{2,\alpha}$-norm for some $\alpha\in(0,1)$.
Indeed, arguing by contradiction as before, the proof reduces to show that
$h_n^\pm\to 0$ in $H^{-\frac12}(\Gamma)$. Since $\Phi_n$ converge now to $I$
with respect to the $C^{2,\alpha}$-norm, we deduce by standard elliptic estimates that
$y_n\to 0$ in $C^{2,\alpha}$-norm up to $\Gamma$, so that 
$\nabla_{\Gamma}(u^\pm_{\Phi_n}\circ\Phi_n)\to \nabla_{\Gamma}u^\pm$
in $C^{1,\alpha}(\overline\Gamma;\R^N)$. 
As $J_{\Phi_n}(D_{\Gamma}\Phi_n)^{-1}(D_{\Gamma}\Phi_n)^{-T}$ converge to $I$ in $C^{1,\alpha}$-norm, we have that $\psi_n^\pm\to 0$
in $C^{1,\alpha}(\overline\Gamma;\R^N)$, hence $\tilde \varphi_n\psi_n^\pm\to 0$
in $H^1(\Gamma;\R^N)$, which implies $h_n^\pm\to 0$ in $H^{-\frac12}(\Gamma)$.
\end{remark}

\begin{proof}[Proof of Theorem~\ref{thm:suff}]
First of all, we note that it is enough to show that there exist $\delta\in (0, \delta_1)$ and $c>0$ 
such that 
for every $\Phi\in{\mathcal D}_\delta\cap C^{\infty}(\overline U; \overline U)$, with 
$\supp(\Phi-I)\cap \Gamma\subset\subset\Gamma$ and $\Gamma_\Phi\neq\Gamma\cap U$, 
\begin{equation}\label{minP}
F((u,K);U)< F((u_\Phi, K_\Phi); U)-c\|\tilde X_\Phi{\,\cdot\,}\nu_\Phi\|^2_{H^1(\Gamma_\Phi)},
\end{equation}
where, we recall, $K_\Phi=\Phi(K\cap U)$ and $\tilde X_\Phi$
is defined in \eqref{xfi}.  
Indeed, the statement would then follow by approximating in the $C^2$-norm
any $\Phi\in{\mathcal D}_\delta$ with diffeormophisms
having the properties above.

The strategy will be the following.
Given $\Phi\in{\mathcal D}_\delta\cap C^{\infty}(\overline U; \overline U)$ with $\delta\le\delta_1$, 
we consider an admissible flow $\Phi_t$ for $\Gamma$ in $U$ 
which coincides with $I+t\tilde X_\Phi$ in 
the $\delta$-neighbourhood of $\Gamma$. 
Setting $g_\Phi(t):=F((u_{\Phi_t},K_{\Phi_t});U)$, we shall show that
there exist $\delta\le\delta_1$ and $c>0$ such that
\begin{equation}\label{f2}
g_\Phi''(t)>2c\|\tilde X_\Phi{\,\cdot\,}\nu_\Phi\|^2_{H^1(\Gamma_\Phi)} 
\quad \text{for every } t\in[0,1] \text{ and every } \Phi\in{\mathcal D}_\delta\cap C^{\infty}(\overline U; \overline U).
\end{equation}
As $g_\Phi'(0)=0$, condition \eqref{f2} will then imply 
\begin{eqnarray*}
F((u,K);U)& = & g_\Phi(0) \hspace{2mm}= \hspace{2mm} g_\Phi(1)-\int_0^1(1-t)g_\Phi''(t)\, dt
\\
& < & g_\Phi(1)-2c\|\tilde X_\Phi{\,\cdot\,}\nu_\Phi\|^2_{H^1(\Gamma_\Phi)}
\int_0^1(1-t)\, dt
\\
& = & \vphantom{\int}
F((u_\Phi,K_\Phi);U)- c\|\tilde X_\Phi{\,\cdot\,}\nu_\Phi\|^2_{H^1(\Gamma_\Phi)},
\end{eqnarray*}
that is \eqref{minP}.

Let us prove \eqref{f2}. Using \eqref{tnorm}, \eqref{tildeTt}, 
and the fact that $X_{\Phi_t}=\tilde X_\Phi$, so that $Z_{\Phi_t}=0$,   
we have by Remark~\ref{remt} that for every $t\in[0,1]$
\begin{eqnarray}\nonumber
g_\Phi''(t) & = & {}- (T_{\Phi_t} (X_{\Phi_t}{\,\cdot\,}\nu_{\Phi_t}), 
X_{\Phi_t}{\,\cdot\,}\nu_{\Phi_t})_{\sim,\Phi_t}
+ \|X_{\Phi_t}{\,\cdot\,} \nu_{\Phi_t} \|_{\sim,\Phi_t}^2 
\\
& &
\label{eq411}
 \hspace{-1cm}
{}+\int_{\Gamma_{\Phi_t}}
f_t
( -2X^{\parallel}_{\Phi_t} {\,\cdot\,} 
\nabla_{\Gamma_{\Phi_t}} ( X_{\Phi_t} {\,\cdot\,} \nu_{\Phi_t} )
+ \B_{\Phi_t} [ X^{\parallel}_{\Phi_t} , X^{\parallel}_{\Phi_t} ] +H_{\Phi_t}(X_{\Phi_t} {\,\cdot\,} \nu_{\Phi_t})^2
)\, d\hn,
\end{eqnarray}
where we recall that $X^{\parallel}_{\Phi_t}$ stands for $(I-\nu_{\Phi_t}\otimes\nu_{\Phi_t})X_{\Phi_t}$
and $f_t=|\nabla_{\Gamma_{\Phi_t}} u_{\Phi_t}^- |^2 - |\nabla_{\Gamma_{\Phi_t}} u_{\Phi_t}^+|^2
+H_{\Phi_t}$.

As $(u,K)$ satisfies the second order condition \eqref{var2+}, 
it follows from Theorem~\ref{thm:eq1} that $\lambda_1=\lambda_{1,I}<1$.
Hence by Lemma~\ref{lm:Pnorm} there exists $\delta_2\in(0,\delta_1)$ such that
\begin{equation}\label{eq414}
\lambda_{1,\Phi}<\tfrac12(\lambda_1+1)<1
\end{equation}
for every $\Phi\in{\mathcal D}_{\delta_2}$. By taking $\delta_2$ smaller, if needed, we can 
also guarantee that 
\begin{equation}\label{also}
\tfrac12 \|\tilde X_\Phi{\,\cdot\,}\nu_\Phi\|^2_{H^1(\Gamma_\Phi)}
\leq \|X_{\Phi_t}{\,\cdot\,}\nu_{\Phi_t}\|^2_{H^1(\Gamma_{\Phi_t})}\leq 
2\|\tilde X_\Phi{\,\cdot\,}\nu_\Phi\|^2_{H^1(\Gamma_\Phi)}
\end{equation}
for every $\Phi\in{\mathcal D}_{\delta_2}$ and every $t\in[0,1]$.
Using the definition of $\lambda_{1,\Phi}$ and
invoking \eqref{eq48}, we deduce
\begin{eqnarray}
\lefteqn{
- (T_{\Phi_t} (X_{\Phi_t}{\,\cdot\,}\nu_{\Phi_t}), X_{\Phi_t}{\,\cdot\,}\nu_{\Phi_t})_{\sim,\Phi_t}
+ \|X_{\Phi_t}{\,\cdot\,} \nu_{\Phi_t} \|_{\sim,\Phi_t}^2
\geq
(1-\lambda_{1,\Phi_t})\|X_{\Phi_t}{\,\cdot\,} \nu_{\Phi_t} \|_{\sim,\Phi_t}^2}
\nonumber
\\
\label{eq415}
& > & \tfrac12 C_1^{-2}(1-\lambda_1)\|X_{\Phi_t}{\,\cdot\,} \nu_{\Phi_t} \|_{H^1(\Gamma_{\Phi_t})}^2
\geq \tfrac14 C_1^{-2}(1-\lambda_1)\|\tilde X_\Phi{\,\cdot\,}\nu_\Phi\|^2_{H^1(\Gamma_\Phi)},
\end{eqnarray}
where the last two inequalities follow from \eqref{eq414} and \eqref{also}.

Choosing $\delta_2$ smaller, if needed, we also have that
$\nu_\Phi=\nu\circ\Pi_{\Gamma\!_r}+\rho_\Phi$ with $\|\rho_\Phi\|_{C^1}<\frac12$
for every $\Phi\in{\mathcal D}_{\delta_2}$.
As $|\tilde X_\Phi|=|\tilde X_\Phi{\,\cdot\,}(\nu\circ\Pi_{\Gamma\!_r})|$, we deduce that
$$
|\tilde X_{\Phi}|\leq |\tilde X_{\Phi}{\,\cdot\,}\nu_\Phi|
+ |\tilde X_{\Phi}{\,\cdot\,}(\rho_\Phi\circ\Pi_{\Gamma\!_r})|
\leq |\tilde X_{\Phi}{\,\cdot\,}\nu_\Phi|
+ \tfrac12 |\tilde X_{\Phi}|,
$$
hence
\begin{equation}\label{oldP}
|\tilde X_{\Phi} | \le 2 |\tilde X_{\Phi}{\,\cdot\,}\nu_\Phi|
\end{equation}
for every $\Phi\in {\mathcal D}_{\delta_2}$.
Moreover, as the $C^{1,\alpha}$-norm of $u_\Phi^\pm$ on $\Gamma_\Phi$ is uniformly
bounded with respect to $\Phi\in{\mathcal D}_{\delta_2}$, one can show that the map 
$$
\Phi\in {\mathcal D}_{\delta_2}
\mapsto \||\nabla_{\Gamma_{\Phi}} u_{\Phi}^- |^2 - |\nabla_{\Gamma_{\Phi}} u_{\Phi}^+|^2
+H_{\Phi}
 \|_{L^\infty(\Gamma_\Phi)}
$$ 
is continuous. In particular, as it vanishes at $\Phi=I$, for every $\e>0$ there esists $\delta\in(0,\delta_2)$ such that
$$
 \| |\nabla_{\Gamma_{\Phi}} u_{\Phi}^- |^2 - |\nabla_{\Gamma_{\Phi}} u_{\Phi}^+|^2
+H_{\Phi}
 \|_{L^\infty(\Gamma_\Phi)}
<\e
$$
for every $\Phi\in {\mathcal D}_\delta$.
Hence, there exists a constant $c_0>0$ such that for every $t\in[0,1]$
\begin{eqnarray*}
\lefteqn{
\int_{\Gamma_{\Phi_t}} \hspace{-2mm}
f_t
( -2X^{\parallel}_{\Phi_t} {\,\cdot\,} 
\nabla_{\Gamma_{\Phi_t}} ( X_{\Phi_t} {\,\cdot\,} \nu_{\Phi_t} )
+ \B_{\Phi_t} [ X^{\parallel}_{\Phi_t} , X^{\parallel}_{\Phi_t} ] +H_{\Phi_t}(X_{\Phi_t} {\,\cdot\,} \nu_{\Phi_t})^2
)\, d\hn }
\\
& \geq & {} 
-2\e\,  \|X_{\Phi_t} \|_{L^2(\Gamma_{\Phi_t})}  \|\nabla_{\Gamma_{\Phi_t}} (X_{\Phi_t}{\,\cdot\,} \nu_{\Phi_t} ) \|_{L^2(\Gamma_{\Phi_t})} -c_0\e\, \|X_{\Phi_t} \|_{L^2(\Gamma_{\Phi_t})}^2 
\\
& \geq & {} \vphantom{\int}
-4(1+c_0)\e\, \|X_{\Phi_t}{\,\cdot\,} \nu_{\Phi_t} \|^2_{H^1(\Gamma_{\Phi_t})}
\geq -8(1+c_0)\e\, \|\tilde X_\Phi{\,\cdot\,}\nu_\Phi\|^2_{H^1(\Gamma_\Phi)},
\end{eqnarray*}
where the last two inequalities follow from \eqref{also}, \eqref{oldP}, 
and the fact that $X_{\Phi_t}=\frac1t \tilde X_{\Phi_t}$. Choosing $\e$ so small that
$8C_1^2(1+c_0)\e<\frac18 (1-\lambda_1)$, 
claim \eqref{f2} follows from the previous inequality, \eqref{eq415}, and \eqref{eq411}, with $c:=\frac{1}{16}C_1^{-2} (1-\lambda_1)$.
\end{proof}

\begin{remark}\label{rmk:add}
We observe that in the course of the proof of Theorem~\ref{thm:suff} we made use of the technical assumption $N\leq3$ only in Lemma~\ref{lm:Pnorm}. 
Thus, by Remark~\ref{rmk:N=4} the following weaker version of
Theorem~\ref{thm:suff} holds in dimension $N>3$. 
If $(u,K)$ is a critical point in $U$ with respect to $\Gamma$ satisfying \eqref{var2+}, then for every $\alpha\in(0,1)$
there exists $\delta>0$ such that 
$$
\int_U |\nabla u|^2\, dx +\hn(K\cap U) <
\int_U |\nabla v|^2\, dx +\hn(\Phi(K\cap U))
$$
for every $C^{2,\alpha}$-diffeomorphism $\Phi$ on $\overline U$ with
$\Phi=I$ on $(K\cap U)\setminus\Gamma$, $\Gamma_\Phi\neq \Gam$, and  $\|\Phi - I\|_{C^{2,\alpha}}\le\delta$, and 
every $v\in L^{1,2}(U\setmeno \Phi(K\cap U))$ with $v=u$ $\hn$-a.e.\ 
on $\partial U$.
In other words, $(u,K)$ is an isolated $C^{2,\alpha}$-local minimizer in $U$ with respect to~$\Gamma$ for any $\alpha\in(0,1)$.
\end{remark}
\end{section}

\begin{section}{Stability and instability results}\label{sec:st}

We start with two results of stability in small domains.
In the first proposition we show that $(u,K)$ is an isolated $C^2$-local minimizer 
in a tubular neighbourhood $(\Gamma)_\e$ of $\Gamma$ (see \eqref{ingra}
for the definition of $(\Gamma)_\e$),
provided condition \eqref{hyp1} is satisfied.

\begin{proposition}\label{thm:small}
Let $\Om$, $(u,K)$, and $\Gamma$ be as in Definition~\ref{flux}.
Assume that $(u,K)$ is a critical point in $\Om$ with respect to $\Gamma$
and that $(\varphi,\varphi)_\sim> 0$ for every $\varphi\in H^1_0(\Gamma)\setmeno\{0\}$.
Assume furthermore that  $(\Gamma)_\e$ is an admissible subdomain for $(u,K)$ (in the sense of 
Definition~\ref{def:partreg}) for every $\e\in(0,\e_0)$.
Then there exists $\e_1\in (0, \e_0)$ such that for every $\e\leq\e_1$
the second variation is positive in $(\Gamma)_\e$; i.e.,
$$
\delta^2\!F((u,\Gamma); (\Gamma)_\e)[\varphi]>0
$$
for every $\varphi\in H^1_0(\Gamma)\setmeno\{0\}$. 
In particular, if $N\leq3$, $(u,K)$ is an isolated $C^2$-local minimizer in $(\Gamma)_\e$ with respect to $\Gamma$, while
if $N>3$, $(u,K)$ is an isolated $C^{2,\alpha}$-local minimizer in $(\Gamma)_\e$ with respect to $\Gamma$ for any $\alpha\in(0,1)$.
\end{proposition}

\begin{remark}
If $N=2$, by the previous proposition and Remark~\ref{rmk:2d} it follows that every critical point
is an isolated $C^2$-local minimizer in a tubular neighbourhood of a compact subarc of the 
regular part of the discontinuity set. This is in agreement with the result in \cite{MM01}, 
where in fact a stronger minimality property is proved. 
Instead if $N\geq3$, there exist critical points whose second variation is nonpositive 
in every tubular neighbourhood of the regular part of the discontinuity set. This follows from
Remark~\ref{rmk:2d}, where it is shown that condition \eqref{hyp1} may fail. 
\end{remark}

\begin{proof}[Proof of Proposition~\ref{thm:small}]
By Theorem~\ref{thm:eq2} it is enough to show that 
\begin{equation}\label{fail}
\lim_{\e\to 0^+}\mu((\Gamma)_\e)=+\infty,
\end{equation}
where $\mu((\Gamma)_\e)$ is the value of \eqref{dual} with $U$ replaced by
$(\Gamma)_\e$.
Assume by contradiction that \eqref{fail} fails. Then there exist $C>0$,
$\e_n\to 0^+$, and $v_n\in L^{1,2}_0
((\Gamma)_{\e_n}\setmeno K; \partial (\Gamma)_{\e_n})$ such that 
$\|R(A^*_+v_n^+ - A^*_-v_n^-)\|_{\sim}=1$ and
$$
\int_{(\Gamma)_{\e_n}}|\nabla v_n|^2\, dx\leq C.
$$
By setting $v_n=0$ on $\Om\setminus(\Gamma)_{\e_n}$ 
we have that $v_n$ is a bounded sequence in $L^{1,2}_0(\Om \setminus K;\partial \Om)$.
Since the measure of $(\Gamma)_{\e_n}$ goes to zero, we deduce that
$v_n$ converge to $0$ weakly in $L^{1,2}_0(\Om\setmeno K;\partial \Om)$.
As the operator \eqref{opera} is compact, we conclude that $R(A^*_+v_n^+ - A^*_-v_n^-)$
converge to $0$ strongly in $H^1_0(\Gamma)$, which contradicts
$\|R(A^*_+v_n^+ - A^*_-v_n^-)\|_{\sim}=1$.

The last part of the statement follows from Theorem~\ref{thm:suff}
and Remark~\ref{rmk:add}.
\end{proof}

In the next proposition we prove that the generic critical point $(u,K)$ is stable
with respect to $C^2$ perturbations with small support.

\begin{proposition}\label{thm:smsupp}
Let $\Om$, $(u,K)$, $U$, and $\Gamma$ be as in Definition~\ref{flux} and
assume in addition that $(u,K)$ is a critical point in $U$ with respect to $\Gamma$. 
Then there exists $R>0$ such that
\begin{equation}\label{task}
\delta^2\!F((u,\Gamma); U)[\varphi]>0
\end{equation}
for every $\varphi\in H^1_0(\Gam)\setmeno\{0\}$ with ${\rm diam}\,(\supp\,\varphi)<R$.
\end{proposition}

\begin{remark}
Arguing as in the proof of Theorem~\ref{thm:suff}, one can show that the thesis 
of Proposition~\ref{thm:smsupp} implies the following minimality property
in dimension $N\leq3$: 
there exists $\delta>0$ such that 
$$
\int_U |\nabla u|^2\, dx +\hn(K\cap U) <
\int_U |\nabla v|^2\, dx +\hn(\Phi(K\cap U))
$$
for every $\Phi\in \mathcal{D}_\delta$  with
${\rm diam}\,(\supp\,(\Phi-I)\cap \Gamma)<R$ and $\Gamma_\Phi\neq\G\cap U$, and 
every $v\in L^{1,2}(U\setmeno \Phi(K\cap U))$ with $v=u$ $\hn$-a.e.\ on $\partial U$.
\end{remark}

\begin{proof}[Proof of Proposition~\ref{thm:smsupp}]
As an easy consequence of Poincar\'e inequality, we infer that there exists $R_0>0$ 
such that for every $x\in\Gamma'$, with $\Gam\subset\subset\Gamma'\subset\subset\Gamma_r$, we have
\begin{equation}\label{psok}
(\varphi,\varphi)_\sim> 0
\end{equation} 
for every $\varphi\in H^1_0(\Gam)\setmeno\{0\}$ with $\supp\,\varphi\subset B_{R_0}(x)$. By Proposition~\ref{prop:ps} the bilinear form \eqref{bilf}
defines an equivalent scalar product on the subspace
$$
H_{x,r}:=\{\varphi\in H^1_0(\Gam):\ \supp\,\varphi\subset B_r(x)\}
$$
for every $r\leq R_0$ and every $x\in\Gamma'$.
Thus we can define by duality the operator $T_{x,r}:H_{x,r}\to H_{x,r}$ satisfying
\begin{equation}\label{tildeTR}
(T_{x,r}\varphi,\psi)_\sim = -2 \int_\Gam (\vf^+\divt (\psi\nablat u^+)
-\vf^-\divt (\psi\nablat u^-))\,d\hn
\end{equation}
for every $\varphi,\psi\in H_{x,r}$. The operator $T_{x,r}$ may be thought of as a ``localization"
of $T$ and turns out to be compact and self-adjoint. We note that by the representation formula
\eqref{tildeTR},  if $B_{r_1}(x_1)\subset B_{r_2}(x_2)$, then for every $\varphi\in H_{x_1,r_1}$ 
the function $T_{x_1,r_1}\varphi$ coincides with the orthogonal projection (with respect to
$(\cdot,\cdot)_\sim$) of $T_{x_2,r_2}\varphi$ on $H_{x_1,r_1}$.
Moreover, using \eqref{psok} and arguing as in the proof of Theorem~\ref{thm:eq1},
one can see that inequality \eqref{task} is satisfied for every $\varphi\in H_{x,r}$ if and only if 
$$
\max_{\genfrac{}{}{0pt}2{\scriptstyle\varphi\in H_{x,r}}{\scriptstyle\|\varphi\|_\sim=1}} (T_{x,r}\varphi,\varphi)_\sim<1.
$$
Therefore, to conclude the proof it is enough to show that
\begin{equation}\label{enth}
\lim_{r\to 0^+} \max_{\genfrac{}{}{0pt}2{\scriptstyle\varphi\in H_{x,r}}{\scriptstyle
\|\varphi\|_\sim=1}}  (T_{x,r}\varphi,\varphi)_\sim = 0
\quad \text{uniformly with respect to }x\in\Gam. 
\end{equation}
Assume by contradiction that \eqref{enth} fails. Then there exist $C>0$,
$x_n\in \Gam$, $r_n\to 0^+$, and $\varphi_n\in H_{x_n,r_n}$ such that $\|\varphi_n\|_\sim=1$ and
\begin{equation}\label{nume0}
(T_{x_n,r_n}\varphi_n, \varphi_n)_\sim \geq C.
\end{equation}
Without loss of generality we can assume that $x_n\to x\in\overline\Gam$ and $B_{r_n}(x_n)\subset
B_{R_0}(x)$ for $n$ large enough.
In particular, by the projection property mentioned before this implies
\begin{equation}\label{nume}
(T_{x_n,r_n}\varphi_n, \varphi_n)_\sim=(T_{x,R_0}\varphi_n, \varphi_n)_\sim.
\end{equation}
As $\|\varphi_n\|_\sim=1$ and the measure of the support of $\varphi_n$ tends to zero, we conclude that 
$\varphi_n\wto 0$ weakly in $H^1_0(\Gam)$.
Since $T_{x,R_0}$ is compact, it follows that
$T_{x,R_0}\varphi_n\to 0$ strongly in $H^1_0(\Gam)$ and in turn,
$(T_{x,R_0}\varphi_n, \varphi_n)_\sim\to 0$. By \eqref{nume} this contradicts \eqref{nume0}.
\end{proof}

We conclude this section with an example of instability in large domains.
A related explicit example will be discussed in the next section.
Let $w:\R^{N-1}\to\R$ be an affine function.
We consider as critical point the pair $(u,K)$, where for every $x=(x',x_N)\in\R^N$ 
$$
u(x):=
\begin{cases}
\hphantom{-}w(x') & \text{for } x_N\geq0, \\
-w(x') & \text{for } x_N<0,
\end{cases}
$$
and $K=\Gamma_r=\{x_N=0\}$.

\begin{proposition}\label{thm:fail}
There exists $R_0>0$ such that the second variation $\delta^2\!F( (u,K\cap B_R(x)); B_R(x))$
is nonpositive for every $x\in K$ and every $R>R_0$.
\end{proposition}

\begin{proof}
We first note that, as $a=0$ in this case, condition \eqref{hyp1} is satisfied.
Therefore, we can consider the operator $T$ defined in \eqref{tildeT}. By Proposition~\ref{conjpt}
there exists a nontrivial solution $(v,\varphi)\in L^{1,2}_0(B_1(0)\setmeno K; \partial B_1(0))
{\times}H^1_0(K\cap B_1(0))$ of \eqref{auxil0} with $U=B_1(0)$ and $\lambda=\lambda_1(B_1(0))$.
For every $x_0\in K$ and every $r>0$
let us consider the functions $v_r\in  L^{1,2}_0(B_r(x_0)\setmeno K; \partial B_r(x_0))$ 
and $\varphi_r\in H^1_0(K\cap B_r(x_0))$ defined by $v_r(x):=v(\frac{x-x_0}{r})$ 
and $\varphi_r(x')=r\varphi(\frac{x'-x'_0}{r})$.
It is easy to see that $(v_r,\varphi_r)$ is a nontrivial solution of \eqref{auxil0} 
with $U=B_r(x_0)$ and $\lambda=r\lambda_1(B_1(0))$. 
Therefore, by Proposition~\ref{conjpt} we have $\lambda_1(B_r(x_0))\geq r\lambda_1(B_1(0))$. 
The conclusion follows by Theorem~\ref{newthm} choosing $R_0=
1/\lambda_1(B_1(0))$. 
\end{proof}

\end{section}

\begin{section}{An explicit example}\label{sec:ex}

As a final application of the results of the previous sections,
we discuss an explicit example, for simplicity in dimension $2$. 
In $\Om=\R^2$ we consider the function
$$
u(x,y)=
\begin{cases}
\hphantom{-}x & \text{for } y\geq0, \\
-x & \text{for } y<0,
\end{cases}
$$
whose discontinuity set is given by $K=\Gamma_r=\R{\times}\{0\}$.
For every Lipschitz bounded domain $U$ in $\R^2$ we recall that $\lambda_1(U)$
denotes the constant introduced in \eqref{max}, corresponding to this choice of $u$, $K$, $U$, and to
$\Gamma=\Gamma_r\cap U$. We will compute explicitly the value of $\lambda_1$ for
rectangles which are symmetric with respect to $K$.

\begin{proposition}\label{prop:x-x}
Let $U=(x_0,x_0+\ell){\times}(-y_0,y_0)$, $y_0>0$,  and let $\Gamma=(x_0,x_0+\ell){\times}\{0\}$.
Then
\begin{equation}\label{defg}
\lambda_1(U)=\tfrac{2\ell}{\pi}\tanh\tfrac{2\pi y_0}{\ell},
\end{equation}
so that the second variation is positive if and only if
\begin{equation}\label{tanh}
\tfrac{2\ell}{\pi}\tanh\tfrac{2\pi y_0}{\ell}<1.
\end{equation}
In particular, if \eqref{tanh} holds,
then $(u,K)$ is an isolated $C^2$-local minimizer in $U$ with respect to $\Gamma$; if 
$\tfrac{2\ell}{\pi}\tanh\tfrac{2\pi y_0}{\ell}>1$,
then $(u,K)$ is not a minimizer in $U$.
\end{proposition}

\begin{remark}\label{rmk:ex}
In \cite{MM01} it is proved that, if a condition stronger than \eqref{tanh}
is satisfied, then a stronger minimality property holds. More precisely, from the results
of \cite[Section~4]{MM01} it follows that
there exists a constant $c_0<1$ such that
if $\tfrac{4\ell}{\pi}\tanh\tfrac{\pi y_0}{\ell}<c_0$, then $(u,K)$ minimizes $F$
among all competitors in $SBV(\Om)$, whose extended graph is contained in a sufficiently small tubular neighbourhood of the extended graph of $u$.
\end{remark}

\begin{proof}[Proof of Proposition~\ref{prop:x-x}]
 We choose $\nu(x,0):=(0,1)$ as an orientation for $\Gamma$. Let $(v,\varphi)$ be a nontrivial solution of \eqref{auxil}. By symmetry we have
$v(x,y)=v(x,-y)$; thus, setting $R:=(x_0, x_0+\ell){\times}(0,y_0)$, we have that $(v,\varphi)$ solves the problem
$$
\begin{cases}
\Delta v =0 & \text{in } R, \\
v=0 & \text{on }\partial R\setmeno\Gamma, \\
\lambda\,\partial_y v =\varphi'
& \text{on }\Gamma, \\
 \varphi''=-4\,\partial_x v
& \text{on }\Gamma.
\end{cases}
$$
Combining together the two conditions on $\Gamma$, we deduce that
$$
\lambda\,\partial_y v=-4(v-c) \quad \text{on } \Gamma,
$$
where $c:=\tfrac{1}{\ell} \int_{x_0}^{x_0+\ell} v(x,0)\, dx$. The computation of $\lambda_1$ amounts to 
the identification of the largest $\lambda$ such that there exists a nontrivial solution of
\begin{equation}\label{Rsys}
\begin{cases}
\Delta v =0 & \text{in } R, \\
v=0 & \text{on }\partial R\setmeno\Gamma, \\
\lambda\,\partial_y v=-4(v-c)
& \text{on }\Gamma.
\end{cases}
\end{equation}
Expanding $v$ in series of sines and taking into account the first two conditions of the system,
we have that
$$
v(x,y)=\sum_n c_n\sin(\tfrac{n\pi}{\ell}(x-x_0))\sinh(\tfrac{n\pi}{\ell}(y_0-y))
$$
with $c_n\in\R$. Differentiating with respect to $y$ and imposing that
$\partial_y v$ has zero average on $\Gamma$, we obtain the condition
\begin{equation}\label{1c}
\sum_{\genfrac{}{}{0pt}2{\scriptstyle n\in \N}{\scriptstyle n \text{ odd}}} c_n\cosh(\tfrac{n\pi y_0}{\ell})=0.
\end{equation}
Expanding also $c$ in series of sines on $[x_0, x_0+\ell]$, one can see that the last condition in \eqref{Rsys}
is equivalent to
$$
\begin{array}{c}
\displaystyle
\lambda\frac{\pi}{\ell}\sum_{n\in\N} nc_n \cosh(\tfrac{n\pi y_0}{\ell})\sin(\tfrac{n\pi}{\ell}(x-x_0))
\smallskip
\\
\displaystyle
=
4\sum_{n\in\N} c_n \sinh(\tfrac{n\pi y_0}{\ell})\sin(\tfrac{n\pi}{\ell}(x-x_0))
-16c \sum_{\genfrac{}{}{0pt}2{\scriptstyle n\in \N}{\scriptstyle n \text{ odd}}} 
\tfrac{1}{n\pi}\sin(\tfrac{n\pi}{\ell}(x-x_0)),
\end{array}
$$
which implies
\begin{eqnarray}
& \label{2c}
\lambda\frac{\pi}{\ell}nc_n \cosh\tfrac{n\pi y_0}{\ell} =
4c_n \sinh\tfrac{n\pi y_0}{\ell} \quad \text{ for } n \text{ even},
\\
& \label{3c}
 \lambda\frac{\pi}{\ell}nc_n \cosh\tfrac{n\pi y_0}{\ell} =
4c_n \sinh\tfrac{n\pi y_0}{\ell}-16c
\tfrac{1}{n\pi} 
\quad \text{ for } n \text{ odd}.
\end{eqnarray}
From \eqref{2c} we deduce that either $c_n=0$ for every $n$ even or there exists an even number
$\bar n$ such that
$$
\lambda=4\tfrac{\ell}{\pi\bar n}\tanh\tfrac{\bar n\pi y_0}{\ell}.
$$
Clearly the biggest $\lambda$ which falls in the latter case, corresponds to $\bar n=2$ and hence,
\begin{equation}\label{lam0}
\lambda_1(U)\geq\tfrac{2\ell}{\pi}\tanh\tfrac{2\pi y_0}{\ell}.
\end{equation}

If $c_n=0$ for every $n$ even, it follows from \eqref{1c} and \eqref{3c} that $c\neq0$.
Hence \eqref{3c} is equivalent to
$$
c_n=\frac{16}{n\pi}\,\frac{c}{4 \sinh\tfrac{n\pi y_0}{\ell}-\lambda\frac{n\pi}{\ell} \cosh\tfrac{n\pi y_0}{\ell}}
$$
for every $n$ odd.
Condition \eqref{1c} and the fact that $c\neq0$ finally yield
\begin{equation}\label{serie}
\sum_{\genfrac{}{}{0pt}2{\scriptstyle n\in \N}{\scriptstyle n \text{ odd}}} 
\frac{1}{n^2}\,\frac{1}{\frac{4}{n} \tanh\tfrac{n\pi y_0}{\ell}-\lambda\frac{\pi}{\ell}} =0.
\end{equation}
By \eqref{lam0} the proof is concluded if we show that the previous equation has no solution in the interval $(\tfrac{2\ell}{\pi}\tanh\tfrac{2\pi y_0}{\ell}, +\infty)$.
If $\lambda>\tfrac{4\ell}{\pi}\tanh\tfrac{\pi y_0}{\ell}$, all the terms of the series in \eqref{serie}
are negative (since $\tanh x/x$ is decreasing for $x>0$), so that we can restrict our attention to the interval $[\tfrac{2\ell}{\pi}\tanh\tfrac{2\pi y_0}{\ell},\tfrac{4\ell}{\pi}\tanh\tfrac{\pi y_0}{\ell})$.
Let $g(\lambda)$ be the function given by the left-hand side of \eqref{serie}. It is easy to see that $g$
is monotone increasing in $[\tfrac{2\ell}{\pi}\tanh\tfrac{2\pi y_0}{\ell},\tfrac{4\ell}{\pi}\tanh\tfrac{\pi y_0}{\ell})$.
Hence it will be enough to prove that $g(\tfrac{2\ell}{\pi}\tanh\tfrac{2\pi y_0}{\ell})>0$. This is equivalent to
\begin{equation}\label{pap}
\frac{1}{4 \tanh\tfrac{\pi y_0}{\ell} -2\tanh\tfrac{2\pi y_0}{\ell}}
> \sum_{\genfrac{}{}{0pt}2{\scriptstyle n\geq3}{\scriptstyle n \text{ odd}}} 
\frac{1}{n^2}\, 
\frac{1}{2\tanh\tfrac{2\pi y_0}{\ell} -\frac{4}{n} \tanh\tfrac{n\pi y_0}{\ell}}.
\end{equation}
Using the inequality
$$
2\tanh\tfrac{2\pi y_0}{\ell} -\frac{4}{n} \tanh\tfrac{n\pi y_0}{\ell}\geq 2\tanh\tfrac{2\pi y_0}{\ell} -\frac{4}{3} \tanh\tfrac{3\pi y_0}{\ell}
$$
for every $n\geq 3$ and the identity
$$
\sum_{\genfrac{}{}{0pt}2{\scriptstyle n\geq3}{\scriptstyle n \text{ odd}}} 
\frac{1}{n^2} =\frac{\pi^2}{8}-1,
$$
inequality \eqref{pap} will be proved if we show 
\begin{equation}\label{papo}
\frac{\frac12\tanh\tfrac{2\pi y_0}{\ell} -\frac13 \tanh\tfrac{3\pi y_0}{\ell}}{\tanh\tfrac{\pi y_0}{\ell} -\frac12\tanh\tfrac{2\pi y_0}{\ell}}
> \frac{\pi^2}{8}-1.
\end{equation}
Applying the addition formula for the hyperbolic tangent it is easy to see that 
$$
\frac{\frac12\tanh(2x) -\frac13 \tanh(3x)}{\tanh x -\frac12\tanh(2x)}=
\frac{5-\tanh^2x}{3(1+3\tanh^2x)}
$$
for every $x>0$. By this identity it is then clear that the left-hand side of \eqref{papo}
is a decreasing function of $y_0$ and its infimum  
is equal to $\frac13>\frac{\pi^2}{8}-1$. 
This concludes the proof of \eqref{pap}, and in turn of \eqref{defg}.

The last part of the statement follows now from Theorem~\ref{thm:eq1}, Remark~\ref{rmk:2d},
Theorem~\ref{thm:suff}, and Theorem~\ref{newthm}.
\end{proof}

\end{section}

\begin{section}{Appendix}\label{appendix}

In this section we collect some auxiliary results, which are needed in the proof of Theorem~\ref{th:var2}.

We start with a proposition where the regularity properties of the map $t\mapsto u_{\Phi_{t}}$ are
investigated (see \eqref{uphi} for the definition of $u_{\Phi_{t}}$). We give only a sketch of the proof.

\begin{proposition}\label{prop:reg}
Under the assumptions of Theorem~\ref{th:var2},
let $\tilde u_t:=u_{\Phi_{t}}\circ\Phi_{t}$ and $ v_t:= \tilde u_t -u$.  The following properties hold:
\begin{itemize}
\item[(i)] the map $t\mapsto  v_t$ belongs to $C^{\infty}((-1,1); L^{1,2}_0(U\setmeno K; \partial U) )$;
\item[(ii)] for every $x_0\in \Gamma$ let $B$ be a ball centered at $x_0$ such that
$B\subset U$, $\overline B\cap \Gamma_s=\varnothing$, and $B\setmeno \Gamma$ has two connected components, $B_+$ and  $B_-$. For every $t\in(-1,1)$ let $\tilde u_{t}^\pm$ denote the restriction of $\tilde u_{t}$ to $B_\pm$. Then the map $\hat u^\pm(t,x):=
\tilde u_{t}^\pm(x)$ belongs to $C^{\infty}((-1,1)\times \overline B_\pm)$.
\end{itemize}
\end{proposition}

\begin{proof}[Proof (Sketch)]  
In order to prove part (i), it is enough to show that for every $t_0\in (-1,1)$ the map $t\mapsto  v_t$
is smooth in a neighbourhood $(t_0-\e, t_0+\e)$. For simplicity we consider only the case $t_0=0$ 
(the general case can be treated similarly). 

First of all, we note that by \eqref{uphi} the function $v_{t}$ solves 
\begin{equation}\label{solves}
\int_U A_{t}[\nabla v_{t}, \nabla z]\, dx+\int_UA_{t}[\nabla u, \nabla z]\, dx=0
\quad \text{for every }z\in L^{1,2}_0(U\setmeno K; \partial U),
\end{equation}
where $A_{t}:=\frac{D\Psi_{t}D\Psi_{t}^T}{{\det D\Psi_{t}}}\circ \Phi_{t}$ and $\Psi_{t}:=\Phi_{t}^{-1}$. 

Let us consider the map $\mathcal{F}: (-1,1){\times}L^{1,2}_0(U\setmeno K; \partial U)
\to L^{1,2}_0(U\setmeno K; \partial U)$ defined in the following way:
for every $t\in(-1,1)$ and every $v\in L^{1,2}_0(U\setmeno K; \partial U)$ the function
$\mathcal{F}(t, v)$ is the unique solution $\xi\in  L^{1,2}_0(U\setmeno K; \partial U)$ of
$$
\int_U\nabla \xi{\,\cdot\,} \nabla z=\int_U A_{t}[\nabla v, \nabla z]\, dx+\int_UA_{t}[\nabla u, \nabla z]\, dx
\quad \text{for every }z\in L^{1,2}_0(U\setmeno K; \partial U).
$$
It can be checked that $\mathcal{F}$ is of class $C^{\infty}$, $\mathcal F(0,0)=0$
(as $(u,K\cap U)\in\Ar(U)$ by assumption), and $\partial_v\mathcal F(0,0)$ is an invertible bounded linear operator from $L^{1,2}_0(U\setmeno K; \partial U)$ onto itself.  
Hence, since $v_{t}$ satisfies $\mathcal{F}(t, v_{t})=0$ by \eqref{solves},
part (i) of the statement follows from the Implicit Function Theorem.

Let us fix $x_0\in\Gamma$ and let $B$, $B_+$, and $B_-$ be as in part (ii)
of the statement.
Let $v'_{t_0}$ be the derivative of $t\mapsto v_{t}$ with respect to the
$L^{1,2}$-norm, evaluated at some $t_0$, which exists by part (i).
We claim that 
\begin{equation}\label{claim-app}
v'_{t_0}=\dot{v}_{t_0} \quad\text{in $B$}\qquad\text{and}\qquad  
\hat{u}^\pm\in C^1((-1,1){\times}\overline
B_{\pm}).
\end{equation}
To this aim we first observe that by \eqref{solves} the function
$w_h:=\frac1h(v_{t_0+h}-v_{t_0})$ is the solution of 
\begin{equation}\label{solves2}
\int_U A_{t_0+h}[\nabla w_h, \nabla z]\, dx
+\int_U\tfrac1h(A_{t_0+h}-A_{t_0})[\nabla \tilde u_{t_0}, \nabla z]\, dx=0
\quad \text{for every }z\in L^{1,2}_0(U\setminus K; \partial U).
\end{equation}
By standard elliptic estimates for every $p>1$ the restrictions $w_h^\pm$ 
to $B_\pm$ satisfy
$$
\|w_h^\pm\|_{W^{2,p}(B_{\pm})}\leq C_p
$$
for some constant $C_p$ independent of $h$. 
We deduce that $v^\pm_{t_0+h}\to v^\pm_{t_0}$ and
$w^\pm_h\to (v'_{t_0})^\pm$ in $C^1(\overline B_{\pm})$, as $h\to 0$. 
In particular, 
\begin{equation}\label{cont1}
(t,x)\mapsto \nabla v^\pm_{t}(x)\quad\text{is continuous in }
(-1,1){\times}\overline B_{\pm} 
\end{equation}
and the equality in \eqref{claim-app} holds.
Moreover, from \eqref{solves2} and the strong convergence of $\nabla w_h$ to
$\nabla \dot{v}_{t_0}$, we infer that
$$
\int_U A_{t_0}[\nabla \dot{v}_{t_0}, \nabla z]\, dx
+\int_U\dot{A}_{t_0}[\nabla \tilde u_{t_0}, \nabla z]\, dx=0\quad 
\text{for every }z\in L^{1,2}_0(U\setmeno K; \partial U).
$$
Using this equation and arguing as before, we obtain
$$
\tfrac1h(\dot v^\pm_{t_0+h}-\dot v^\pm_{t_0})\to 
\ddot{v}_{t_0}^\pm\quad\text{in $C^0(\overline B_{\pm})$},
$$
which yields, in particular, the continuity of the map 
$(t,x)\mapsto \dot v^\pm_{t}(x)$.
Together with \eqref{cont1}, this implies that the map 
$(t,x)\mapsto v^\pm_{t}(x)$ belongs to $C^1((-1,1){\times}\overline B_{\pm})$,
which is equivalent to the second part of \eqref{claim-app}.
Finally, the $C^{\infty}$ regularity can be obtained by iterating the
arguments above.
\end{proof}

The content of the next lemma is a pair of preliminary identities, which will be needed
in the proof of Lemma~\ref{lem:identita}.

\begin{lemma}
Under the assumptions of Theorem~\ref{th:var2}, the following equalities hold on $\Gamma$:
\begin{eqnarray} 
& DX [X^{\parallel}, \nu] =-
\B[X^{\parallel}, X^{\parallel}] 
+  X^{\parallel} {\,\cdot\,} 
\nabla_{\Gamma} ( X {\,\cdot\,} \nu),   \label{gradx} \\
&\frac{\partial}{\partial t}( \nu_{\Phi_{t}} \circ \Phi_{t} \big)|_{t=0}
= DX[ \nu , \nu ] \nu - ( DX )^{T} [ \nu ]=-(D_{\Gamma}X)^T[\nu]. \label{nufipto}
\end{eqnarray}
\end{lemma}

\begin{proof}
As $X^{\parallel} {\,\cdot\,} \nabla_{\Gamma} ( X {\,\cdot\,} \nu )=X^{\parallel} {\,\cdot\,} \nabla ( X {\,\cdot\,} \nu ) = ( DX )^T [ \nu, X^{\parallel} ] + ( D \nu )^T [ X, X^{\parallel} ]$,
identity \eqref{gradx} follows by observing that $(D\nu)^T=(D_{\Gamma}\nu)^T=\B$ on $\Gamma$
and that $T_x\Gamma$ is invariant for $D_{\Gamma}\nu(x)$ for every $x\in\Gamma$.

Setting  $w_{t}:=(D\Phi_{t} )^{-T}  [\nu]$, it follows from \eqref{nutfit} that
\begin{equation}
\frac{\partial}{\partial t}( \nu_{\Phi_{t}} \circ \Phi_{t} )|_{t=0}
= - ( \nu {\,\cdot\,} \dot{w} ) \nu + \dot{w}. \label{nuph}
\end{equation}
The equality \eqref{nufipto} then follows from \eqref{nuph} and the fact that $\dot{w}= -( D X )^{T}[ \nu ]$.
\end{proof}

We conclude this appendix with the proof of Lemma~\ref{lem:identita}.

\begin{proof}[Proof of Lemma~\ref{lem:identita}]
To simplify the notation in the sequel we will write simply $u$ instead of $u^\pm$.
For $x\in \Gamma$ let  $\tau_1(y) , \dots , \tau_{N-1}(y)$ denote an orthonormal basis of $T_y\Gamma$
which varies smoothly with $y$ in  a neighbourhood of $x$.
For $i \in \{ 1, \dots, N-1 \}$ we have
$$
\partial_{\tau_i}( \partial_{\tau_i} u) =\partial_{\tau_i} ( \nabla u {\,\cdot\,} \tau_i ) =
\nabla^2 u\, [ \tau_i , \tau_i ] + \nabla u {\,\cdot\,} \partial_{\tau_i} \tau_i.
$$
Expressing $\partial_{\tau_i} \tau_i$ in the basis $\{\tau_1 , \dots , \tau_{N-1} , \nu\}$
and using the fact that $\partial_{\nu} u=0$ on $\Gamma$, we obtain
$$
\partial_{\tau_i}( \partial_{\tau_i} u) =  \nabla^2 u\, [ \tau_i , \tau_i ] +
\sum_{k=1}^{N-1} (\tau_k {\,\cdot\,} \partial_{\tau_i} \tau_i  ) \partial_{\tau_k} u. 
$$
Hence, as $\tau_k {\,\cdot\,} \partial_{\tau_i} \tau_i=- \tau_i {\,\cdot\,} \partial_{\tau_i} \tau_k$ and
$\tau_i {\,\cdot\,} \tau_k=0$ for $k\neq i$, we have
\begin{eqnarray*}
\sum_{i=1}^{N-1} \nabla^2 u\,[\tau_i, \tau_i] 
& = &
\sum_{i,k=1}^{N-1} (\tau_i {\,\cdot\,} \tau_k)
 \partial_{\tau_i} ( \partial_{\tau_k} u) 
+ \sum_{i,k=1}^{N-1} ( \tau_i {\,\cdot\,} \partial_{\tau_i} \tau_k )
\partial_{\tau_k} u 
\\
& = & \sum_{i=1}^{N-1} \tau_i {\,\cdot\,} \partial_{\tau_i} 
\Big( \sum_{k=1}^{N-1} \partial_{\tau_k}u\, \tau_k \Big) 
= \Delta_{\Gamma} u.
\end{eqnarray*}
Since $u$ is harmonic, the first term in the previous identity coincides with $-\nabla^2 u\,[\nu, \nu]$,
so that (a) follows.

By differentiating along the direction $\tau_i$  the identity $\partial_{\nu} u=0$ we deduce
$$
0 = \partial_{\tau_i} ( \nabla u {\,\cdot\,} \nu ) = 
\nabla^2 u\, [ \tau_i , \nu ]
+ \nabla u {\,\cdot\,} \partial_{\tau_i} \nu
= \nabla^2 u\, [ \tau_i , \nu ]
+ \nabla_{\Gamma} u {\,\cdot\,} \partial_{\tau_i} \nu.
$$
Since $\partial_{\tau_i} \nu=\B\tau_i$ and $\B$ is symmetric, the previous equality yields
$$
\nabla^2 u\, [ \tau_i , \nu ] 
=-\B[ \nabla_{\Gamma} u ,  \tau_i]\qquad \text{for }i=1,\dots, N-1.
$$
By linearity the identity continues to hold if $\tau_i$ is replaced by any tangent vector.
Hence, writing $\nabla^2 u\, [ X , \nu ] = ( X {\,\cdot\,} \nu )\, \nabla^2 u [ \nu , \nu ] 
+ \nabla^2 u\, [ X^{\parallel} , \nu ] $ and applying (a), we have
\begin{equation}\label{bpf}
\nabla^2 u\, [ X , \nu ] = -( X {\,\cdot\,} \nu ) \Delta_{\Gamma} u
-\B[ \nabla_{\Gamma} u , X^{\parallel} ] = -( X {\,\cdot\,} \nu ) \Delta_{\Gamma} u
-\B[ \nabla_{\Gamma} u , X ],
\end{equation}
where in the last equality we used the fact that $\B[ \nabla_{\Gamma} u]$ is tangent to $\Gamma$.
This proves (b). We also note that identity \eqref{bpf} still holds when $X$ is replaced 
by $\nabla_{\Gamma}u$ (in fact by any vector field), so that we obtain (e). 

Using (b) and recalling that $\B=D_{\Gamma} \nu$, we find 
$$
\div_{\Gamma}[ ( X {\,\cdot\,} \nu ) \nabla_{\Gamma} u ] 
= (D_{\Gamma} X )^T [ \nu , \nabla_{\G} u ]
+\B[ X , \nabla_{\Gamma} u ] 
+ ( X {\,\cdot\,} \nu ) \Delta_{\Gamma} u  = (D_{\Gamma} X )^T [ \nu , \nabla_{\G} u ] 
- \nabla^2 u\, [ X , \nu ],
$$
which shows (c).

Since $D\nu$ coincides with the Hessian of the signed distance function,
we have by \cite[Theorem~3, Part~I]{AD00} that
$$
\partial_{\nu} (D\nu) = -( D\nu )^2.
$$
Since $H=\div\,\nu$ and $D\nu=D_{\Gamma}\nu=\B$ is symmetric on $\Gamma$, 
we immediately deduce (d).

As $\frac{\partial}{\partial t}( \nu_{\Phi_{t}} \circ \Phi_{t} )|_{t=0}
= \dot{\nu} +D\nu [ X ]$, we obtain (f) by comparison
with (\ref{nufipto}).

Finally, as $\frac{\partial}{\partial t}( J_{\Phi_{t}} )|_{t=0} = \div_{\Gamma} X$
(see \cite[Lemma 2.49]{Zo}), we have by (\ref{nufipto}) 
\begin{eqnarray*}
\tfrac{\partial}{\partial t}(\dot{\Phi}_{t} {\,\cdot\,} ( \nu_{\Phi_{t}} \circ \Phi_{t} ) \, J_{\Phi_{t}}  )|_{t=0}
& = &
Z {\,\cdot\,} \nu + X {\,\cdot\,} \tfrac{\partial}{\partial t}( \nu_{\Phi_{t}} \circ \Phi_{t} )|_{t=0}
+ ( X {\,\cdot\,} \nu ) \div_{\Gamma} X
\\
& = & 
Z {\,\cdot\,} \nu + ( X {\,\cdot\,} \nu ) DX [ \nu , \nu ]
- ( D X )^{T} [ \nu , X ] 
+ ( X {\,\cdot\,} \nu ) \div_{\Gamma} X  
\\
& = &
Z {\,\cdot\,} \nu - ( D X )^T [  \nu, X^{\parallel}  ]+ ( X {\,\cdot\,} \nu ) \div_{\Gamma} X 
\\
& = & 
Z {\,\cdot\,} \nu - DX [ X^{\parallel} , \nu ]
- X^{\parallel} {\,\cdot\,} \nabla_{\Gamma} ( X {\,\cdot\,} \nu  )
+ \div_{\Gamma} \big( ( X {\,\cdot\,} \nu ) X \big).
\end{eqnarray*}
Using (\ref{gradx}) we obtain identity (g).
\end{proof}

\end{section}

\bigskip
\bigskip
\bigskip
\bigskip
\noindent
{\bf Acknowledgments.} {
The authors wish to thank Gianni Dal Maso for interesting discussions on the subject of the paper.
This work is part of the Project ``Calculus of Variations" 2004, 
supported by the Italian Ministry of Education, University, and Research and of the research project 
``Mathematical Challenges in Nanomechanics" sponsored by Istituto Nazionale di Alta Matematica (INdAM) ``F.~Severi".
}

\bigskip
\bigskip
\bigskip
\bigskip
{
}

\end{document}